\documentclass[12pt,reqno]{amsart}
\usepackage{amssymb,mathrsfs,graphicx}
\usepackage{mathtools}
\usepackage[margin=3cm]{geometry}
\mathtoolsset{showonlyrefs=true}

\title[Eulerian dynamics with radial symmetry]{
  Eulerian dynamics in multi-dimensions with radial symmetry}

\author[Changhui Tan]{Changhui Tan}
\address[Changhui Tan]{\newline Department of Mathematics, \ 
 University of South Carolina, 1523 Greene St., Columbia, SC 29208, USA}
\email{tan@math.sc.edu}
\thanks{\textit{Acknowledgment.} This work has been supported by the NSF grant
DMS 1853001. }



\newtheorem{theorem}{Theorem}[section]
\newtheorem{lemma}[theorem]{Lemma}

\newtheorem{proposition}[theorem]{Proposition}
\newtheorem{remark}[theorem]{Remark}
\newtheorem{definition}{Definition}[section]

\def\R{\mathbb{R}}
\def\pa{\partial}
\def\F{\mathbf{F}}
\def\e{\mathbf{e}}
\def\u{\mathbf{u}}
\def\x{\mathbf{x}}
\def\y{\mathbf{y}}
\def\z{\mathbf{z}}

\def\grad{\nabla}
\def\div{\grad\cdot}

\def\psim{\underline{\psi}}
\def\tr{\text{tr}}
\def\psiM{{\psi_M}}
\def\psim{\nu}

\begin{document}
\allowdisplaybreaks

\begin{abstract}
 We study the global wellposedness of pressure-less Eulerian dynamics in
 multi-dimensions, with radially symmetric data.
 Compared with the 1D system, a major difference in
 multi-dimensional Eulerian dynamics is the presence of the \emph{spectral
 gap}, which is difficult to control in general.
 We propose a new pair of scalar quantities that provides a
 significant better control of the spectral gap. Two applications are
 presented. (i) the Euler-Poisson equations: we show a \emph{sharp}
 threshold condition on initial data that distinguish global
 regularity and finite time blowup; (ii) the Euler-alignment equations: we
 show a large subcritical region of initial data that leads to
 global smooth solutions. 
\end{abstract}

\maketitle 

\vspace{-.2in}
{\small\tableofcontents}

\section{Introduction}\label{sec:intro}
We consider the following pressure-less Euler equation with forces
\begin{align}
& \pa_t\rho+\div(\rho\u)=0,\label{eq:density}\\
& \pa_t(\rho\u)+\div(\rho\u\otimes\u)=\rho\F,\label{eq:momentum}
\end{align}
subject to the initial condition
\begin{equation}\label{eq:init}
  \rho(\x,t=0)=\rho_0(\x),\quad \u(\x,t=0)=\u_0(\x).
\end{equation}
Here, $\rho:\R^n\times\R_+\to\R$ represents the density of the fluid,
and $\u:\R^n\times\R_+\to\R^n$ is the flow velocity.
$\F$ is a general forcing acting on the flow. It could depend on $\rho$ and $\u$.

The Eulerian dynamics \eqref{eq:density}-\eqref{eq:momentum}
is a fundamental system of equations in fluid mechanics.
It has a vast amount of applications with different choices of 
forces $\F$.
A big challenging and demanding question is to understand whether the
solutions are globally regular, or there could be singularity
formations in finite time.

\subsection{Spectral dynamics and the spectral gap}
The momentum equation \eqref{eq:momentum} can be equivalently written as the
following dynamics of the velocity $\u$, in the non-vacuous region
\begin{equation}\label{eq:velo}
\pa_t\u+(\u\cdot\grad)\u=\F.
\end{equation}
When $\F\equiv0$, \eqref{eq:velo} is the classical \emph{inviscid
  Burgers equation}.
It is well-known that the solution admits a finite time shock
formation, for any generic smooth initial data.
Indeed, in one dimension, taking $x$-derivative of the equation,
one immediately obtain
$(\pa_t+u\pa_x)(\pa_xu)=-(\pa_xu)^2$.
This yields a Ricatti equation of $\pa_xu$ along the characteristic
paths, which governs the main structure of the solution:
blowup happens in finite time if initially $\pa_xu_0(x)<0$.
The idea of tracing the dynamics of $\pa_xu$ also works very well for
1D models of the type \eqref{eq:velo}, with different forcing terms.

In multi-dimensions, taking the spatial gradient of \eqref{eq:velo} would yield 
\begin{equation}\label{eq:gradu}
  (\pa_t+\u\cdot\grad)\grad\u=-(\grad\u)^{\otimes2}+\grad\F,
\end{equation}
where the velocity gradient $\grad\u$ is an $n$-by-$n$ matrix.
In many applications, the boundedness of $\grad\u$ plays a crucial
role in the propagation of the regularity of the solution.
A natural question would be
\begin{center}
\emph{Which scalar quantities exhibit the same Ricatti structure  as
$\pa_xu$ in 1D?}
\end{center}

One candidate is the set of eigenvalues of $\grad\u$, denoted by
$\{\lambda_i\}_{i=1}^n$.
Indeed, when $\F\equiv0$, the dynamics of $\lambda_i$, known as the
\emph{spectral dynamics}, satisfies the same
Ricatti equation as 1D: $(\pa_t+\u\cdot\grad)\lambda_i=-\lambda_i^2$.
It can be solved explicitly along the characteristic paths,
deducing a similar blowup phenomenon,
despite of the fact that $\lambda_i$ could be complex-valued.

With the forcing term, the spectral dynamics of \eqref{eq:gradu} has
the form
\begin{equation}\label{eq:lambF}
  (\pa_t+\u\cdot\grad)\lambda_i=-\lambda_i^2+l_i^T(\grad\F) r_i,\quad i=1,\cdots,n,
\end{equation}
where $(l_i, r_i)$ are the corresponding left and right eigenvectors
of $\lambda_i$.
It has been studied extensively in \cite{liu2002spectral}.
Although one can largely benefit from the
explicit Ricatti structure, it is in general hard to control
$l_i^T(\grad\F) r_i$, as in many cases $\grad\F$ does not share the
same eigenvectors with $\grad\u$.

Another natural replacement of $\pa_xu$ in multi-dimensions would be
\emph{the divergence}
\[d:=\div\u=\text{tr}(\grad\u)=\sum_{i=1}^n\lambda_i,\]
whose dynamics can be obtained by taking the trace of
\eqref{eq:lambF}. It reads
\[(\pa_t+\u\cdot\grad)d = -\tr\left((\grad\u)^{\otimes2}\right)+\div\F.\]

Investigating the dynamics of $d$ has a couple of advantages.
First, $d$ is real-valued. More importantly, it is more friendly to
the forcing term, as $\div\F$ is much easier to
handle (compared with $l_i^T(\grad\F) r_i$) in many applications.

However, the term $\tr\left((\grad\u)^{\otimes2}\right)\neq d^2$, for
$n\geq2$. 
The difference is related to the \emph{spectral gap} of the matrix
$\grad\u$, defined as
\begin{equation}\label{eq:spectralgap}
  \eta =\frac{1}{2}\sum_{i=1}^n\sum_{j=1}^n(\lambda_i-\lambda_j)^2.
\end{equation}
Indeed, it is easy to check that the difference
\begin{equation}\label{eq:diff}
  d^2-\tr\left((\grad\u)^{\otimes2}\right)=\frac{n-1}{n}d^2-\frac{1}{n}\,\eta.
\end{equation}
Therefore, to make use of the Ricatti structure and to
extend 1D regularity results to multi-dimensions, one needs to
additionally control the spectral gap, which turns out to be a
difficult task. As we will argue in
Remark \ref{rem:divnogood}, $\div\u$ might not be a good replacement
for $\pa_xu$, due to the presence of the spectral gap.

In the following, we focus on two classical models on Eulerian dynamics with
nonlocal interaction forces.

\subsection{The Euler-Poisson equations}
The Euler-Poisson equations is a fundamental system in plasma
physics. It describes the electron fluid interacting with its own
electric field against a charged ion background \cite{guo1998smooth}.
The pressure-less Euler-Poisson equations have the form
\eqref{eq:density}-\eqref{eq:momentum}, with the force
\begin{equation}\label{eq:EPforce}
  \F=-\kappa\grad(-\Delta)^{-1}(\rho-c),
\end{equation}
where the parameter $\kappa$ denotes the strength of the charge force,
and $c\geq0$ is a constant background.

The 1D Euler-Poisson equation has been studied extensively in
\cite{engelberg2001critical}, where a sharp critical threshold on the
initial data is obtained that distinguishes the global
wellposedness of solutions and the finite-time singularity formations.
The result is extended to the system with pressure in \cite{tadmor2008global}.

However, in higher dimensions, global wellposedness remains to be a
challenging open problem.
In the case where pressure is presented, global solutions can be
obtained for small initial data perturbed from a constant state
\cite{guo1998smooth,jang2014smooth}, leveraging the dispersive
structure. For the pressureless system, very little is known, even for
small initial data.
The main difficulty on the spectral analysis \eqref{eq:lambF} is that
\[\grad\F=-\kappa\grad\otimes\grad(-\Delta)^{-1}(\rho-c)\]
is a nonlocal Reisz transform on $\rho$, which is hard to control.

An important observation is that, $\div\F=\kappa(\rho-c)$ depends only
on local information of $\rho$. Therefore, the force is more friendly
when tracing the dynamics of the divergence
\[d' = -\tr\left((\grad\u)^{\otimes2}\right)+\kappa(\rho-c).\]
This approach has been studied in \cite{tadmor2003critical}.
Although the forcing term is much easier to handle, the major difficulty is
shifted to the control of the spectral gap \eqref{eq:spectralgap},
which depends non-locally on $\rho$ and $d$.
A restricted Euler-Poisson (REP) equation is introduced in
\cite{tadmor2003critical}, with modifications on the
$\tr\left((\grad\u)^{\otimes2}\right)$ term so that the spectral gap
becomes locally dependent on $\rho$.
However, the result can not be easily extended to the Euler-Poisson
equations due to the lack of control on the spectral gap.

\subsection{The Euler-alignment equations}
Another model of Eulerian dynamics is called the Euler-alignment
system, where
\begin{equation}\label{eq:EAforce}
  \F=\int\phi(\x-\y)(\u(\y,t)-\u(\x,t))\rho(\y,t)\,d\y.
\end{equation}
It is the macroscopic representation of the Cucker-Smale model
\cite{cucker2007emergent}, describing the emergent behavior in animal
flocks.
The $\F$ is called a nonlocal alignment force, where $\phi$ is the
influence function that measures the strength of the influence
between a pair of agents.
The Euler-alignment system was first introduced and formally derived
in \cite{ha2008particle}, with rigorous justifications in
\cite{figalli2018rigorous}.

The Euler-alignment system has been studied in
\cite{tadmor2014critical}. The result contains threshold conditions
on initial data which leads to global regularity or finite-time
singularity formations, in both 1D and 2D. In particular, the 2D
result is obtained by tracing the dynamics of $d$, together with a control
of the spectral gap. The conditions are not sharp, due to the
non-locality of the alignment force.

In a successive work \cite{carrillo2016critical}, a remarkable
commutator structure in $\F$ was discovered, which leads to a sharp
critical threshold that distinguishes global regularity and finite
time blowup of the solutions, for the system in 1D.
It also reveals intriguing connections to other models in fluid
mechanics. Then, theories on global solutions are developed in 1D for
different types of influence functions, including strongly singular
alignment \cite{do2018global,shvydkoy2017eulerian}, weakly singular
alignment \cite{tan2020euler}, as well as misalignment
\cite{miao2020global}. Different behaviors are observed in each
case. In particular, with strongly singular alignment, the system
becomes dissipative, and all smooth
non-vacuous initial data leads to global regularity.
All 1D results are sharp.

For the multi-dimensional Euler-alignment system, much less is known in
regards to global regularity. In \cite{he2017global}, improved
threshold conditions are derived in 2D, taking advantage of the commutator
structure, which turns out to be the same as 1D in the dynamics of
$d$. However, the result is far from optimal, as one needs to
additionally control the spectral gap.
With strongly singular alignment, global regularity is proved in
\cite{shvydkoy2019global}, for small initial data near the steady
state. The result is much weaker than 1D. The smallness condition is
used to control the spectral gap.
\medskip

The two models above are two examples of Eulerian dynamics, where
the global regularity theory is much less developed in
multi-dimensions, compared with one-dimension.
The major difficulty is to \emph{control the effect of the spectral
  gap}.

In this paper, we study the Eulerian dynamics with radially symmetric
initial data.
Despite the redial symmetry, the effect of the spectral gap still
persists (see \eqref{eq:sgrelation}).
We propose a new pair of scalar quantities as the replacement of
$\pa_xu$ in 1D. Compared with $\{\lambda_i\}_{i=1}^n$, the quantities
are real-valued, and are more friendly to the forcing term. Compared
with $d$, the dynamics of the quantities have the precise Ricatti
structure, so we can avoid a direct nonlocal control of the spectral
gap. 

The newly proposed quantities allow us to obtain significantly better
regularity results for Eulerian dynamics in multi-dimensions with
radial symmetry. We apply the idea to the Euler-Poisson and the
Euler-alignment equations. Further extension can be made to a large
class of Eulerian dynamics with different forcing terms.

For the Euler-Poisson equations, we obtain a \emph{sharp} threshold
condition, stated in Theorem \ref{thm:EP}.
This is the first sharp result on the Euler-Poisson system
in multi-dimensions for all smooth radially symmetric initial data
(see Remark \ref{rem:EPsharp} for more discussions).
For the Euler-alignment equations, we show global regularity with a
large region of initial data,  in Theorem \ref{thm:EA}.
Although the result is not sharp, it significantly improves
the existing results in the literature (see Remark \ref{rem:EAimprove}).

The rest of the paper is organized as follows. In section
\ref{sec:radial}, we introduce the new scalar quantities, and state
our main results.
We will then discuss the Euler-Poisson equations and the
Euler-alignment equations in sections \ref{sec:EP} and
\ref{sec:EA} respectively.
We end the paper with some further discussion in section \ref{sec:conclusion}.

\section{Radially symmetric solutions and the new scalar quantities}\label{sec:radial}
We focus on a special type of solutions for the Eulerian dynamics
\eqref{eq:density}-\eqref{eq:momentum}, with
radial symmetry and without swirl
\begin{equation}\label{eq:radial}
  \rho(\x,t)=\rho(r,t),\quad\u(\x,t)=\frac{\x}{r}u(r,t).
\end{equation}
Here, $r=|\x|\in\R_+$ is the radial variable. $\rho$ and  $u$ are scalar
functions defined in $\R_+\times\R_+$.
Appropriate boundary conditions at $r=0$ are assumed, to ensure regularity
of $\rho$ and $\u$ at the origin, for instance, $u(0,t)=0$, $\pa_r\rho(0,t)=0$.
In all examples that we concern, the force takes the form
\begin{equation}\label{eq:forceradial}
  \F(\x,t)=\frac{\x}{r}F(r,t).
\end{equation}
So, the radial symmetry is preserved in time.

Our goal is to find appropriate scalar quantities that serve as the
multi-dimensional replacement of $\pa_xu$ that exhibit the Ricatti
structure, and meanwhile help us control the spectral gap $\eta$ as
well as the force $F$.

Let us first calculate the divergence
\begin{equation}\label{eq:divrelation}
  d=\div\u=u_r+(n-1)\frac{u}{r},
\end{equation}
and the difference in \eqref{eq:diff} (representing the spectral gap $\eta$)
\begin{equation}\label{eq:sgrelation}
  d^2-  \tr\left((\grad\u)^{\otimes2}\right)=
  2(n-1)u_r\frac{u}{r}+(n-1)(n-2)\frac{u^2}{r^2}.
\end{equation}
Clearly, the term in \eqref{eq:sgrelation} does not vanish in the
radially symmetric setup, and can not be determined by
local information in $d$.

A remarkable observation is that, both the divergence and the
difference can be determined by \emph{local} information of the two
quantities $u_r$ and $\frac{u}{r}$.
In fact, the spectral gap $\eta=\frac{1}{n-1}(u_r-\frac{u}{r})^2$.

Hence, we propose to use the pair
\begin{equation}\label{eq:pair}
  (p, q):=\left(u_r, \frac{u}{r}\right)
\end{equation}
as the multi-dimensional replacement of $\pa_xu$.

Note that the boundedness of the pair \eqref{eq:pair} is equivalent to
the boundedness of $\grad\u$, the quantities that play a crucial
role in preserving the regularity of the solution.
\begin{proposition}\label{prop:equivgradu}
  Suppose $u_r$ and $\frac{u}{r}$ are bounded. Then, $\grad\u$ is bounded.
\end{proposition}
\begin{proof}
  It follows from the direct computation
  \[\pa_{x_j}u_i=\frac{x_ix_j}{r^2}u_r+\left(\delta_{ij}-\frac{x_ix_j}{r^2}\right)\frac{u}{r},\]
  where $\delta_{ij}$ is the Kronecker delta.
\end{proof}

In the following, we argue that the pair \eqref{eq:pair} is a better
replacement of $\pa_xu$, compared with $\{\lambda_i\}_{i=1}^n$ and $d$.
We proceed with four examples: the inviscid Burgers equation, the
damped Burgers equation, the Euler-Poisson equations, and the
Euler-alignment equations.

\subsection{The inviscid Burgers equation}
Consider the inviscid Burgers equation \eqref{eq:velo} with
$\F\equiv0$
\[\u_t+(\u\cdot\grad)\u=0,\]
under the radially symmetric setup \eqref{eq:radial}. 
The dynamics of the pair \eqref{eq:pair} reads
\[
  \begin{cases}
    p'=-p^2,\\
    q'=-q^2,
  \end{cases}
\]
where $'=\pa_t+u\pa_r$ denotes the material derivative.
It is a decoupled system, with two Ricatti equations the same as
\eqref{eq:lambF}. This immediately implies a sharp global regularity
result.

\begin{theorem}
  The solution of the radially symmetric inviscid Burgers equation is
  globally regular, if and only if
  \[u^0_r(r)\geq0,\quad\text{and}\quad \frac{u^0(r)}{r}\geq0.\]
\end{theorem}
\begin{proof}
  The Ricatti structure implies that $(p,q)$ are uniformly bounded in
  time if and only if the condition holds. Global regularity then
  follows from Proposition~\ref{prop:equivgradu} and the classical
  equivalency between boundedness of $\grad\u$ and global regularity.
\end{proof}

\begin{remark}
From \eqref{eq:divrelation}, we know the divergence $d$ is a linear
combination of $(p,q)$. However, due to the nonlinear evolution of
$(p,q)$, we have
\[d'=p'+(n-1)q'=-p^2-(n-1)q^2\neq -d^2,\]
and the difference \eqref{eq:sgrelation} can not be expressed locally
in terms of $d$, and additional nonlocal control is required on the
spectral gap.
This indicates the advantage of studying the pair $(p,q)$ compared
with the divergence $d$.
\end{remark}

\subsection{The damped Burgers equation}
Let us consider another example \eqref{eq:velo}, with a damping force
$\F=-\kappa\u$. This corresponds to the damped Burgers
equation
\begin{equation}\label{eq:dampedBurgers}
  \u_t+(\u\cdot\grad)\u=-\kappa\u.
\end{equation}
Similarly, one can obtain the dynamics of the pair \eqref{eq:pair}
under the radial symmetric setup
\[
  \begin{cases}
    p'=-p^2-\kappa p,\\
    q'=-q^2-\kappa q.
  \end{cases}
\]
Solving the decoupled system, we obtain
\begin{theorem}\label{thm:dampedBurgers}
  The radially symmetric solutions of the damped Burgers equation
  \eqref{eq:dampedBurgers} are globally regular, if and only if
  \begin{equation}\label{eq:CTdampedBurgers}
    u^0_r(r)\geq-\kappa,\quad\text{and}\quad
    \frac{u^0(r)}{r}\geq-\kappa.
  \end{equation}
\end{theorem}

\begin{remark}\label{rem:divnogood}
  For the 1D damped Burgers equation, a solution is regular if and only
  if $u_0'(x)\geq-\kappa$ for all $x\in\R$. One would naturally think
  $d_0=\div\u_0\geq-\kappa$ would be the condition in the
  multi-dimensional case. However, this is neither a
  sufficient nor a necessary condition of \eqref{eq:CTdampedBurgers}.
  This is an indication that the divergence $d$ does not serve as a good
  replacement of $\pa_xu$ in the multi-dimensional cases.
\end{remark}

For both inviscid and damped Burgers equations, working directly with
the eigenvalues $\{\lambda_i\}_{i=1}^n$ of $\grad\u$ will yield the same result, even for
general initial data that does not require radial symmetry. This is
due to the simple structure of the force $\F$, which is either $0$ or
$-\kappa\u$. In both cases,
$\grad\F$ shares the same eigenvectors as $\grad\u$, and so the term
$l_i^T(\grad\F)r_i$ in \eqref{eq:lambF} is simply $0$ or
$-\kappa\lambda_i$, respectively.
However, for more general interacting forces, spectral dynamics are
hard to trace, as the forcing term $\grad\F$ is difficult to control.

The new paired quantities in \eqref{eq:pair} have a big advantage in
dealing with general nonlocal forces.
In the following, we focus on the two examples: the Euler-Poisson
equations, and the Euler-alignment equations. We show strong
regularity results for these systems, thanks to our new paired quantities.

\subsection{Main results}
First, we consider the Euler-Poisson equations
\eqref{eq:density}-\eqref{eq:momentum} and \eqref{eq:EPforce} under
the radially symmetric setup \eqref{eq:radial}.
The parameter $\kappa>0$, representing the strength of the repulsive
force. The parameter $c$ can be either zero, or a positive constant,
which corresponds to two scenarios: zero background, and constant
background. The solutions under the two cases are known to have
very different asymptotic behaviors.

\begin{theorem}[Sharp threshold condition for the Euler-Poisson
  equations]\label{thm:EP}
  Consider the Euler-Poisson equation
  \eqref{eq:density}-\eqref{eq:momentum} and \eqref{eq:EPforce}  with
  smooth initial data
  $\rho_0-c\in H^s(\R^n)$ and $\u_0\in H^{s+1}(\R^n)^n$, for
  $s>\frac{n}{2}$, and satisfying the radial symmetry
  \eqref{eq:radial}.
  Then, there exists a region $\Sigma\in\R^4$,
  defined in Definition \ref{def:EPsub}, depending on $n, \kappa, c$, such that
  \begin{itemize}
  \item If the initial condition satisfies
  \begin{equation}\label{eq:EPthreshold}\left(\pa_ru_0(r), \frac{u_0(r)}{r},
      -\frac{\pa_r\phi_0(r)}{r},\rho_0(r)\right)\in\Sigma,
  \end{equation}
   for all $r>0$, then the system admits a global smooth solution
   (in the sense of \eqref{eq:EPregularity}).
   Here, $\phi_0(\x):=(-\Delta)^{-1}(\rho_0(\x)-c)$, which is radially symmetric.
   \item If there exists an $r>0$ such that \eqref{eq:EPthreshold} is
     violated, then the solution blows up in finite time. Moreover,
     the blowup won't happen at $r=0$.
 \end{itemize}
\end{theorem}

\begin{remark}\label{rem:EPsharp}
  The global regularity for multi-dimensional Euler-Poisson equations
  in a challenging problem, even under the radially symmetric setup.
  When pressure is presented and with a non-zero background, global
  solutions are shown in \cite{wang2001global} with the help of additional relaxation.
  In \cite{jang2012two}, global regularity is shown in 2D for small
  initial data, featuring an algebraic decay towards the constant
  steady state. Under the pressure-less setup, to our best knowledge,
  the only regularity result is in \cite{wei2012critical}, where
  a critical threshold condition is shown, only for the zero
  background case $(c=0)$, and with expanding flows $u_0(r)>0$.

  Our result works for both zero and constant background cases. It
  is the very first result that provides a sharp characterization on
  all initial conditions,
  which lead to either global wellposedness or finite time blowup.
  In particular, it covers initial data that is not fully expanding.
  One remarkable and non-trivial discovery is, for any initial data with compression
  ($u_0(r)<0$ so the velocity points to the origin), the Poisson force
  helps to avoid blowup at the origin, so that there won't be
  concentrations at the origin.

  For the zero background case, we derive a more explicit expression of
  the subcritical region $\Sigma$. Like the 1D result in
  \cite{engelberg2001critical}, global regularity can be obtained as
  long as $\pa_ru_0$ is not too negative (see Theorem \ref{thm:EP-prho-sub}).
\end{remark}

Our next result is on the Euler-alignment equations
\eqref{eq:density}-\eqref{eq:momentum} and \eqref{eq:EAforce},
with a bounded Lipschitz influence function $\phi$.

\begin{theorem}[Threshold conditions for the Euler-alignment
  equations]\label{thm:EA}
  Consider the Euler-alignment equation
  \eqref{eq:density}-\eqref{eq:momentum} and \eqref{eq:EAforce}  with
  smooth compact initial data
  $\rho_0\in H_c^s(\R^n)$ and $\u_0\in H^{s+1}(\R^n)^n$, for
  $s>\frac{n}{2}$, and satisfying the radial symmetry
  \eqref{eq:radial}. Denote
  \[G_0(|\x|)=\pa_ru_0(|\x|)+\int_{\R^n}\phi(|\x-\y|)\rho_0(\y)\,d\y,\]
  which is a radially symmetric function. Also, set a constant
  $C_0>0$ that depends on initial data as
  $C_0:=\|\phi'\|_{L^\infty}\|\rho_0\|_{L^1}\|u_0\|_{L^\infty}$.
  Then,
  \begin{itemize}
    \item There exist subcritical thresholds $\sigma_G^+$ and
      $\sigma_q^+$, defined in \eqref{eq:sqp} and \eqref{eq:sGp}
      respectively, such that if the initial data satisfy
      \[G_0(r)\geq\sigma_G^+(C_0)\quad\text{and}\quad
        \frac{u_0(r)}{r}\geq\sigma_q^+(C_0),\quad\forall~r>0,\]
      then, the system admits a global smooth solution. Moreover, the
      solution exhibits the flocking phenomenon \eqref{eq:flocking}
      with fast alignment \eqref{eq:fastalign}.
    \item There exist supercritical thresholds $\sigma_G^-$ and
      $\sigma_q^-$, defined in \eqref{eq:sqm} and \eqref{eq:sGm}
      respectively, such that if there exists an $r>0$ where
      \[G_0(r)<\sigma_G^-(C_0)\quad\text{or}\quad
        \frac{u_0(r)}{r}<\sigma_q^-(C_0),\]
      then the solution blows up in finite time.
    \end{itemize}
  \end{theorem}

  \begin{remark}\label{rem:EAimprove}
    To our best knowledge, this is the first result that provides a
    large subcritical region of initial data that leads to global
    regularity, for the Euler-alignment equations in three (or more)
    dimensions.
    It also provides an enhanced subcritical region in 2D, compared
    with the existing results \cite{tadmor2014critical, he2017global}.

    The thresholds $\sigma_G^\pm$ depend on the dimension $n$. As
    illustrated in Figure \ref{fig:EA-Gn}, in 1D,
    $\sigma_G^+=\sigma_G^-\equiv0$. This recovers the sharp threshold
    condition in \cite{carrillo2016critical}.
    The thresholds $\sigma_q^\pm$ are independent of $n$. As
    illustrated in Figure \ref{fig:EA-q},
    $\sigma_q^+$ is negative when $C_0$ is small. Therefore, the
    subcritical region includes initial data where the flow has compression.

    In the special case when $\phi$ is a constant (say $\phi\equiv1$),
    the Euler-alignment
    equations can be reduced to the damped Burgers equation
    \eqref{eq:dampedBurgers}, with
    $\kappa=\|\rho_0\|_{L^1}$.
    Our threshold conditions become the sharp condition in
    \eqref{eq:CTdampedBurgers}. Indeed, we have $C_0=0$. One can
    further check that $\sigma_G^+(0)=\sigma_G^-(0)=0$ and
    $\sigma_q^+(0)=\sigma_q^-(0)=-\|\rho_0\|_{L^1}$.
  \end{remark}

\section{Application to the Euler-Poisson equations}\label{sec:EP}
In this section, we discuss the pressure-less Euler-Poisson equation
\begin{align*}
  &\pa_t\rho+\div(\rho\u)=0,\quad x\in\R^n,~~t\geq0,\\
  &\pa_t\u+(\u\cdot\grad)\u=-\kappa\grad\phi,\quad-\Delta\phi=\rho-c.
\end{align*}
Here, $\rho$ is the density and $\u$ is the velocity field. $\phi$ is
the electrical charge potential. $c\geq0$ is a constant background.
The parameter $\kappa$ characterizes the strength of the charge
force. We shall focus on the more intriguing case when the force is
repulsive, namely $\kappa>0$.

Let us first state the well-known local wellposedness theory.
\begin{theorem}[Local wellposedness]\label{thm:EPlocal}
  Consider the Euler-Poisson equations with initial data
  $\rho_0-c\in H^s(\R^n)$ and
  $\u_0\in H^{s+1}(\R^n) ^n$, for $s>\frac{n}{2}$.
  Then, there exists a time $T>0$ such that
  the solution
  \begin{equation}\label{eq:EPregularity}
    (\rho,\u)\in C([0,T],  H^{s}(\R^n))\times
    C([0,T], H^{s+1}(\R^n))^n.
  \end{equation}
  Moreover, the life span $T$ can be extended as long as
  \begin{equation}\label{eq:EPBKM}
    \int_0^T\|\grad\u(\cdot,t)\|_{L^\infty}\,dt<+\infty.
  \end{equation}
\end{theorem}

Under the radial symmetry \eqref{eq:radial}, the system can be
expressed as
\begin{align}
  &\rho_t+(\rho u)_r=-\frac{(n-1)\rho u}{r},\label{eq:EPrho}\\
  &u_t+uu_r=-\kappa\phi_r,\label{eq:EPu}\\
  &-\phi_{rr}-(n-1)\frac{\phi_r}{r}=\rho-c.\label{eq:EPphi}
\end{align}

Let us compute the dynamics of the pair \eqref{eq:pair}: $p=u_r$,
$q=\frac{u}{r}$, together with the dynamics of $\rho$ along each
characteristic path
\[
  \begin{cases}
    p'=-p^2-\kappa \phi_{rr}=-p^2+\kappa\left(\rho-c+(n-1)\frac{\phi_r}{r}\right),\\
    q'=-q^2-\kappa \frac{\phi_r}{r},\\
    \rho'=-\rho (p+(n-1)q),
  \end{cases}
\]
where the relation \eqref{eq:EPphi} is used in the second equality of
the dynamics of $p$.

Observe that the dynamics is not a closed system, but with only one
nonlocal term $\frac{\phi_r}{r}$. One way to get rid of the nonlocal
contribution is to seek for cancelations. Indeed, the term goes away
if we evolve the divergence $d$
\[d'=p'+(n-1)q'=(-p^2-(n-1)q^2)+\kappa(\rho-c).\]
This reflects the fact that the divergence is friendly to the forcing
term.
However, one has to bear with the effect of the spectral gap, which is
difficult to control.

Instead, we directly work with the $(p,q,\rho)$ dynamics. Let
\[s:=-\frac{\phi_r}{r}\]
be the extra quantity involved.  To get the dynamics of $s$ along the
characteristic path, we
rewrite \eqref{eq:EPphi} as
\[(-r^{n-1}\phi_r)_r=r^{n-1}(\rho-c).\]
From \eqref{eq:EPrho}, the right hand side $r^{n-1}(\rho-c)$ satisfies
\[
  \pa_t\left(r^{n-1}(\rho-c)\right)+\pa_r\left(r^{n-1}(\rho-c)u\right)
  =-\pa_r\left(cr^{n-1}u\right).
\]
Then, its primitive $e:=-r^{n-1}\phi_r$ would satisfy
\[e_t+ue_r=-cr^{n-1}u.\]
As $s=er^{-n}$, we have
\[s'=e'r^{-n}-nr^{-n-1}r'e=-c\frac{u}{r}-nu\frac{s}{r}=-(c+ns)q.\]

Since the density $\rho\geq0$, we get
\[s=r^{-n}\int_0^r\left(\tau^{n-1}(\rho(\tau)-c)\right)\,d\tau
  \geq r^{-n}\left(-\frac{cr^n}{n}\right)=-\frac{c}{n}.\]
The strict inequality can be achieved for $r>0$ if we assume
$\rho_0(0)>0$.

Thus, we end up with a closed system of $(p,q,s,\rho)$ along each
characteristic path.
\begin{equation}\label{eq:EPfull}
  \begin{cases}
    p'=-p^2+\kappa(\rho-c-(n-1)s),\\
    q'=-q^2+\kappa s,\\
    s'=-(ns+c)q,\\
    \rho'=-\rho(p+(n-1)q).
  \end{cases}
\end{equation}
The global solvability of the PDE system reduces to the decoupled ODE
systems along characteristic paths.

\begin{definition}[Subcritical region]\label{def:EPsub}
Let $\Sigma\in\R^4$ be the set defined as follows:
\[(p_0, q_0, s_0, \rho_0)\in\Sigma\]
if and only if
\begin{itemize}
 \item[(i).] $\rho_0\geq0$ and $s_0>-\frac{c}{n}$.
 \item[(ii).]  the ODE system \eqref{eq:EPfull} with initial condition 
$(p_0, q_0, s_0, \rho_0)$ is bounded globally in time.
\end{itemize}
\end{definition}

Now, we are ready to prove Theorem \ref{thm:EP}.

\begin{proof}[Proof of Theorem \ref{thm:EP}]
First, for subcritical initial data, from the Definition~\ref{def:EPsub},
we know $\pa_ru(r,t)$ and $\frac{u(r,t)}{r}$ are bounded globally in
time. Then, Proposition~\ref{prop:equivgradu} implies boundedness of
$\grad\u$. Finally, condition \eqref{eq:EPBKM} holds for any finite
time $T$, leading to global regularity.

Next, for supercritical initial data, at least one quantity out of
$(p, q, s, \rho)$ should blow up in finite time.
We will show later that $(q,s)$ stays bounded in all time. So, the
blowup can only happen to $p$ or $\rho$.
If $p$ blows up at time $T$, $\grad\u(\cdot,T)$ becomes unbounded,
and consequently $\u(\cdot,T)\not\in (H^{s+1}(\R^n))^n$ for any
$s>n/2$.
If $\rho$ blows up at time $T$, $\rho(\cdot,T)\not\in H^s(\R^n)$.
Therefore, the solution loses regularity \eqref{eq:EPregularity} in finite time.

Finally, we show that blowup won't happen at the origin.
Note that such blowup happens when a characteristic path $r(t)$
starting at $r_0>0$ reaches zero at a finite time. However, we have
\[\frac{d}{dt} r(t)= u(r(t),t) = r(t) q(r(t),t).\]
As $q$ is uniformly bounded in time (we will show this later), we
obtain
\[r(t)\geq r_0e^{-\int_0^t\|q(\cdot,\tau)\|_{L^\infty}d\tau}>0.\]
Hence, blowup can not happen at the origin.
\end{proof}

The rest of the section is devoted to showing that $(q,s)$ are uniformly
bounded, and to providing more explicit
descriptions of the set $\Sigma$.

\subsection{The one-dimensional case}
When $n=1$, the quantities $(q,s)$ do not contributed towards the
dynamics of $(p,\rho)$. The ODE system \eqref{eq:EPfull} reduces to
\[
  \begin{cases}
    p'=-p^2+\kappa(\rho-c),\\
    \rho'=-\rho p,
  \end{cases}
\]
which has been studied in \cite{engelberg2001critical}.
$\Sigma$ can be explicitly expressed by
\begin{equation}\label{eq:Sigma1D}
  \Sigma=
\begin{cases}~~\left\{(p_0,\rho_0)\left|~p_0>-\sqrt{2\kappa\rho_0}\right.\right\}&c=0,\\
&{\tiny~~}\\
~~\left\{(p_0,\rho_0)\left|~|p_0|<\sqrt{\kappa(2\rho_0-c)}\right.\right\}&c>0.
\end{cases}
\end{equation}

\subsection{Multi-dimensional cases with zero background}
For dimensions $n\geq2$, the dynamics of $(p,\rho)$ depends on $(q,s)$.
The coupled quantities serve as the characterization of the spectral
gap effect, which appears only in multi-dimensions.
Since the behaviors of the dynamics are different between $c=0$ and
$c>0$, we shall first discuss the zero
background case.

\subsubsection{Uniform boundedness of $(q,s)$}
We now study the dynamics of $(q,s)$, which form a closed system,
independent of $(p,\rho)$
\begin{equation}\label{eq:qs}
  \begin{cases}
    q'=-q^2+\kappa s,\\
    s'=-nsq.
  \end{cases}
\end{equation}
The main result is summarized as follows.

\begin{theorem}\label{thm:EP_qsc0}
  Let $n\geq2$. Consider the $(q,s)$ dynamics in \eqref{eq:qs} with bounded
  initial conditions $(q_0,s_0)$ such that $s_0>0$.
  Then, $(q(t), s(t))$ remains bounded in all time.
  Moreover, $(q(t),s(t))$ converges to $(0,0)$ as $t\to\infty$.
\end{theorem}

Theorem \ref{thm:EP_qsc0} ensures uniform boundedness of $(q,s)$.
The proof involves non-trivial analysis on
the phase plane of $(q,s)$. 

First, express $s$ in terms of $q$ along the characteristic path as
\begin{equation}\label{eq:st}
  s(t)=s_0\exp\left[-n\int_0^tq(\tau)d\tau\right].
\end{equation}
Clearly, $s$ remains bounded and positive as long as $q$ is bounded.

We start with the relatively easy case when $q_0\geq0$ for every
characteristic path. This
corresponds to expanding waves, as $u_0(r)\geq0$ for all $r\geq0$.
So one does not need to worry about concentration at the origin.

This particular setup has been investigated in
\cite{wei2012critical}, by studying the explicit dynamics of the
characteristic trajectories in time.
The following lemma shows that $q_0\geq0$ is
an invariant region in the phase plane of $(q,s)$.
As illustrated in the curve starting at $A$ in
Figure~\ref{fig:EP_sq_c0}, the trajectory of $(q,s)$ in the phase
plane stays bounded, and is attracted to the steady state $(0,0)$.

\begin{figure}[ht]
  \includegraphics[width=.7\linewidth]{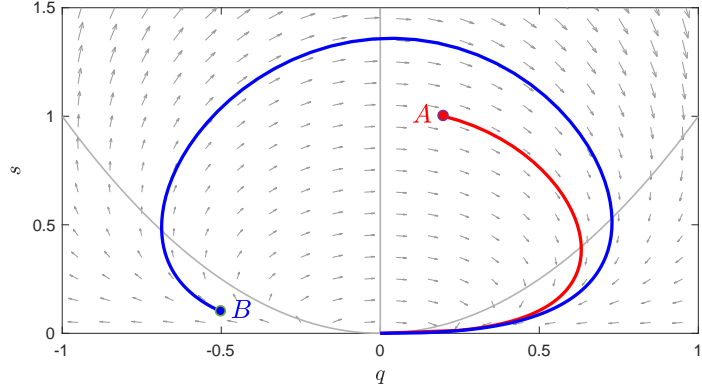}
  \caption{Illustration of the the phase plane of $(q,s)$ with $c=0$ and $N=2$.
  If the initial data $(q_0,s_0)=A$, the trajectory stays in
  $\R_+\times\R_+$, and converges to $(0,0)$. If the initial data
  $(q_0,s_0)=B$, the trajectory stays bounded, and will cross $q=0$ in
  finite time. Asymptotically, it also converges to $(0,0)$.}\label{fig:EP_sq_c0}
\end{figure}

\begin{lemma}\label{lem:qsplus}
Consider the dynamics \eqref{eq:qs} with $q_0\geq0$ and $s_0>0$, 
then $(q,s)$ remains bounded in all time. More precisely, there exists a
constant $Q$  such that
\[q(t)\in[0,Q], \quad s(t)\in(0,s_0],\quad \forall~t\geq0.\]
Moreover, $(q(t),s(t))$ converges to $(0,0)$ as $t\to\infty$.
\end{lemma}
\begin{proof}
  First, we show $q(t)\geq0$ for all $t\geq 0$. Suppose the argument
  is false, then there exists a time $t_0$ such that $q(t_0)=0$ and
  $q'(t_0)\leq0$.
  On the other hand, $q'(t_0)=\kappa s(t_0)>0$, which
  leads to a contradiction.

  Then, by \eqref{eq:st}, we get $s(t)\leq s_0$. Therefore, $s$ is
  bounded.
  
  Next, we claim that $q(t)\leq Q:=\max\{q_0, \sqrt{\kappa
    s_0}\}$, using a similar argument by contradiction as the first
  part. Given any $\epsilon>0$, suppose there exists a $t_0$ such that 
  $q(t_0)=Q+\epsilon$, and $q(t_0+)>Q+\epsilon$. Then,
  \[q'(t_0)=-(Q+\epsilon)^2+\kappa
  s(t_0)\leq-Q^2-\epsilon(2Q+\epsilon)+\kappa
  s_0=-\epsilon(2Q+\epsilon)<0.\]
  Therefore, $q(t_0+)<Q+\epsilon$, which leads to a
  contradiction. The proof is finished by taking $\epsilon\to0$.

  Finally, for the asymptotic behavior, we first observe that $s$ is bounded
  and decreasing and hence has a limit. Since the only steady state
  for $(q,s)\in[0,Q]\times[0,s_0]$ is $(0,0)$,
  $\lim_{t\to\infty}s(t)=0$. For $q$, let
  $t_1=\inf\{t\geq0:s(t)<\frac{Q^2}{\kappa}\}$. Clearly, $t_1$ is
  finite. $q(t)$ is decreasing for $t\geq t_1$ and hence has a
  limit. Due to the only steady state $(0,0)$, the limit has to be 
  $\lim_{t\to\infty}q(t)=0$. This concludes the proof.
\end{proof}

The more subtle case is when $q_0<0$, namely the initial velocity is
pointing towards the origin. Without the term $\kappa s$,
the dynamics $q'=-q^2$ is known to blow up to $-\infty$ in finite
time. The $\kappa s$ term could help avoid the blowup. The following
theorem describes such phenomenon. Remarkably, when $n\geq2$, the
blowup won't happen for any initial configuration, no matter how small
$s_0$ is. 
\begin{theorem}\label{thm:EPb1}
Let $n\geq2$. Consider the dynamics \eqref{eq:qs} with initial data
$q(0)<0$ and $s(0)>0$. Then, $(q,s)$ remains uniformly bounded in all
time. Moreover, $(q(t),s(t))$ converges to $(0,0)$ as $t\to\infty$.
\end{theorem}

We first prove the theorem for dimension $n\geq3$, or in general $n>2$
($n$ does not to be an integer to make sense of the dynamics
\eqref{eq:qs}).

\begin{proof}[Proof of Theorem \ref{thm:EPb1} for $n\geq3$]
  First, we show that $q$ is bounded from below.
  Let us start with a rough estimate on $q$. Since $\kappa s>0$, 
\[q'\geq-q^2.\]
This implies $q(t)\geq \frac{q_0}{1+tq_0}$. Therefore, blowup can not
happen before $T_0=-\frac{1}{q_0}$.

For any $t<T_0$ and $\tau<t$, we have
\[-\frac{1}{q(t)}+\frac{1}{q(\tau)}\geq -(t-\tau)\quad\Rightarrow\quad
q(\tau)\leq\frac{q(t)}{1-q(t)(t-\tau)},\quad\forall~\tau\in[0,t).\]

Apply the estimate to \eqref{eq:st}, we obtain
\begin{equation}\label{eq:sest}
  s(t)\geq
s_0\exp\left[-n\int_0^t\frac{q(t)}{1-q(t)(t-\tau)}d\tau\right]
=s_0(1-tq(t))^n\geq s_0t^n(-q(t))^n.
\end{equation}

Plug back in \eqref{eq:qs}, we get an improve estimate on the dynamics
of $q$
\begin{equation}\label{eq:qest3}
  q'(t)\geq-(-q(t))^2+\kappa s_0 t^n(-q(t))^n.
\end{equation}

Since $n>2$, the second term on the right hand side of
\eqref{eq:qest3} will dominate the first term, and $q'(t)>0$ if
$-q(t)$ is big enough. This avoids $q$ from becoming more negative,
and hence prevents blowup.

In detail, for $t>T_0/2$,
\[  q'(t)>-(-q(t))^2+\frac{\kappa s_0T_0^n}{2^n}(-q(t))^n
  \geq0,\quad\text{if}\quad
  -q(t)\geq\left(\frac{\kappa s_0T_0^n}{2^n}\right)^{-\frac{1}{n-2}}.
\]
This implies
\[  q(t)\geq \min\left\{q\left(\frac{T_0}{2}\right),-\left(\frac{\kappa
        s_0T_0^n}{2^n}\right)^{-\frac{1}{n-2}}\right\},\quad\forall~t>\frac{T_0}{2}.\]

Together with the rough estimate $q(t)\geq 2q(0)$ for all $t\leq
T_0/2$, we end up with a uniform in time lower bound on $q$.

To get a uniform bound on $s$, we argue by contradiction. Suppose $s$
is not uniformly bounded. Since $s$ is bounded in all finite time, we
must have $\lim_{t\to\infty}s(t)=\infty$. By
Lemma \ref{lem:qsplus}, it must be true that $q(t)<0$ for all time,
and hence $s(t)$ is increasing.

On the other hand, there exists a time $t_0$ such that
$s(t_0)>\kappa^{-1}(q_{\min}^2+1)$, where $q_{\min}$ denotes the lower bound of
$q$. Then,
\[q'(t)=-q(t)^2+\kappa s(t)\geq -q_{\min}^2+\kappa s(t_0)>1,\quad
\forall~t\geq t_0.\]
Consequently, we have $q(t)\geq q_{\min}+(t-t_0)$, and then
$q(t_0+(-q_{\min}))\geq0$. This leads to a contradiction.

Therefore, there exists a finite time $t_1$ such that $s(t_1)$ reaches
the maximum, and $q(t_1)=0$. Starting from $t_1$, we can apply
Lemma~\ref{lem:qsplus} and get the upper bound on $q$ and the
asymptotic behaviors.
\end{proof}

The two-dimensional case is \emph{critical}, as the estimate \eqref{eq:qest3} does not
directly imply $q'(t)>0$ for large $-q(t)$, if $s_0$ is
small. To show boundedness of solutions for all $s_0>0$, we need to
make further improvements to our estimates.

\begin{proof}[Proof of Theorem \ref{thm:EPb1} for $n=2$]
  We start with the same argument as $n>2$ case, which implies
  \eqref{eq:sest}
  \begin{equation}\label{eq:qest4}
    q'(t)\geq \big(-1+s_0t^2\big)(-q(t))^2.
  \end{equation}
  Then, $q(t)'>0$ for any $t>s_0^{-1/2}$. Hence, blowup won't happen
  after $T_1=s_0^{-1/2}$. Also, blowup can not happen before
  $T_0=-\frac{1}{q_0}$. Therefore, $q(t)$ is bounded from below in all
  time if $T_0>T_1$, or equivalently $s_0>(-q_0)^2$.
  However, if $s_0$ is small $s_0\leq (-q_0)^2$,
  then blowup can still occur at $t\in(T_0, T_1)$. In this
  scenario, we perform the following improved estimates.

  For any $0\leq\tau<t<T_0$, from \eqref{eq:qest4} we have
\[-\frac{1}{q(t)}+\frac{1}{q(\tau)}\geq
-\left[t-\tau-\frac{s_0}{3}(t^3-\tau^3)\right]
~~\Rightarrow~~
q(\tau)\leq\frac{q(t)}{1-q(t)\left[t-\tau-\frac{s_0}{3}(t^3-\tau^3)\right]}.\]

This leads to an improved estimate on
\[\int_0^tq(\tau)d\tau\geq
\int_0^t\frac{1}{\frac{1}{q(t)}-\tau+\frac{s_0}{3}(\tau^2-3t\tau+3t^2)\tau}d\tau
\geq\int_0^t\frac{1}{\frac{1}{q(t)}-\tau+\frac{s_0}{3}t^2\tau}d\tau
\]
and then
\[s(t)\geq
s_0\exp\left[-2\int_0^t\frac{1}{\frac{1}{q(t)}-(1-\frac{s_0}{3}t^2)\tau}d\tau\right]
=s_0\left[1-\left(1-\frac{s_0}{3}t^2\right)tq(t)\right]^{\frac{2}{1-\frac{s_0}{3}t^2}}.\]

Compared with the estimate \eqref{eq:sest} with
$s(t)\gtrsim(-q(t))^2$, the improved estimate has
$s(t)\gtrsim(-q(t))^\alpha$, with
$\alpha=\frac{2}{1-\frac{s_0}{3}t^2}>2$. Now, we are able to finish
the proof using the same argument as in the $n\geq3$ case.
Indeed, for $t>T_0/2$,
\[  q'(t)>-(-q(t))^2+\kappa s_0\left(1-\frac{s_0T_0^2}{12}\right)(-q(t))^\frac{2}{1-\frac{s_0T_0^2}{12}}
  \geq0,\]
if $-q(t)$ is large enough,
\[ -q(t)\geq\left[\kappa s_0
  \left(1-\frac{s_0T_0^2}{12}\right)\right]^{-\frac{12-s_0T_0^2}{2s_0T_0^2}}.\]

The uniform bound on $s$ and asymptotic behaviors can then be obtained
the same as the $n>2$ case.
\end{proof}

\begin{remark}
  For $n<2$, the dynamics can lead to a finite time blowup if $q_0<0$
  and $s_0$ is small enough. Therefore, $n\geq2$ is a critical
  assumption for Theorem \ref{thm:EPb1} to be valid. We skip the
  discussion for the $n<2$ case, as it is not relevant under our setup. 
\end{remark}

\subsubsection{Asymptotic behavior}
The next lemma shows the detailed asymptotic behavior of $(q,s)$ as time
approaches infinity. The convergence rate will be useful for later discussions.
Without loss of generality, we set $q_0\geq0$. This is because we know
that if $q_0<0$, there exists a finite time $t_*$ such that
$q(t_*)=0$. Same convergence rate can be obtained by a simple shift in time.

\begin{lemma}\label{lem:qsasymp}
  Consider the dynamics \eqref{eq:qs} with $q_0\geq0$ and
  $s_0>0$. Then, there exist two positive constants $C_q$ and $\bar{C}_s$,
  depending on $n$ and $(q_0, s_0)$, such that
  \begin{equation}\label{eq:qsasymp}
    q(t)\leq C_q(t+1)^{-1},\quad s(t)\leq \bar{C}_s(t+1)^{-2}, \quad
    \forall~t\geq0.
  \end{equation}
  Moreover, there exists a positive constant
  $C_s$, depending on $n$ and $(q_0, s_0)$, such that
  \begin{equation}\label{eq:sasymp}
    s(t)\leq \begin{cases}
      C_s(t+1)^{-2}\left(\ln(t+1)+1\right)^{-1}&n=2\\C_s(t+1)^{-n}&n\geq3.
      \end{cases}.
  \end{equation}
\end{lemma}
\begin{proof}
We apply the following transformation. Let
\[\hat{q}(t)=(t+1)q(t),\quad \hat{s}(t)=(t+1)^2s(t).\]
We can rewrite the dynamics \eqref{eq:qs} as
\[  \begin{cases}
    \hat{q}'=\frac{1}{t+1}\left(-\hat{q}^2+\hat{q}+\kappa \hat{s}\right),\\
    \hat{s}'=\frac{1}{t+1}(2-n\hat{q})\hat{s}.
  \end{cases}
\]
and the pre-factor $\frac{1}{t+1}$ can be absorbed by changing the time
variable to $\hat{t}=\ln(t+1)$. So, with respect to $\hat{t}$, the
dynamics reads
\begin{equation}\label{eq:qshat}
  \begin{cases}
    \hat{q}'=\left(-\hat{q}^2+\hat{q}+\kappa \hat{s}\right),\\
    \hat{s}'=(2-n\hat{q})\hat{s}.
  \end{cases}
\end{equation}
From a standard study of the autonomous system in the phase plane (see
Figure \ref{fig:EP_sqhat}), we
know that for any $\hat{q}_0\geq0, \hat{s}_0>0$, the dynamics
converges to the steady state $(1,0)$.
This implies $q(t)=O(t^{-1})$, and $s(t)=o(t^{-2})$.
In particular, we can pick $C_q=\hat{q}_{\max}$.

\begin{figure}[ht]
  \includegraphics[width=.48\linewidth]{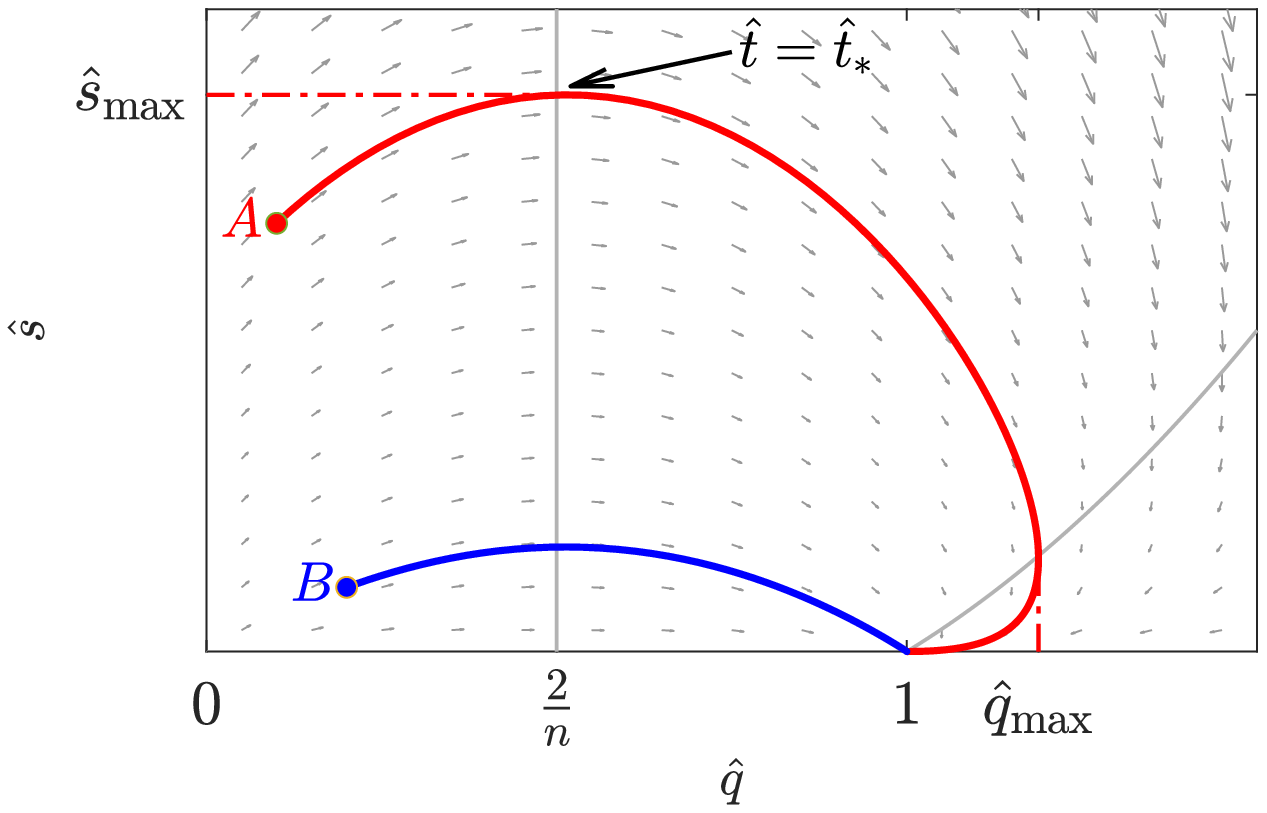}
  \includegraphics[width=.48\linewidth]{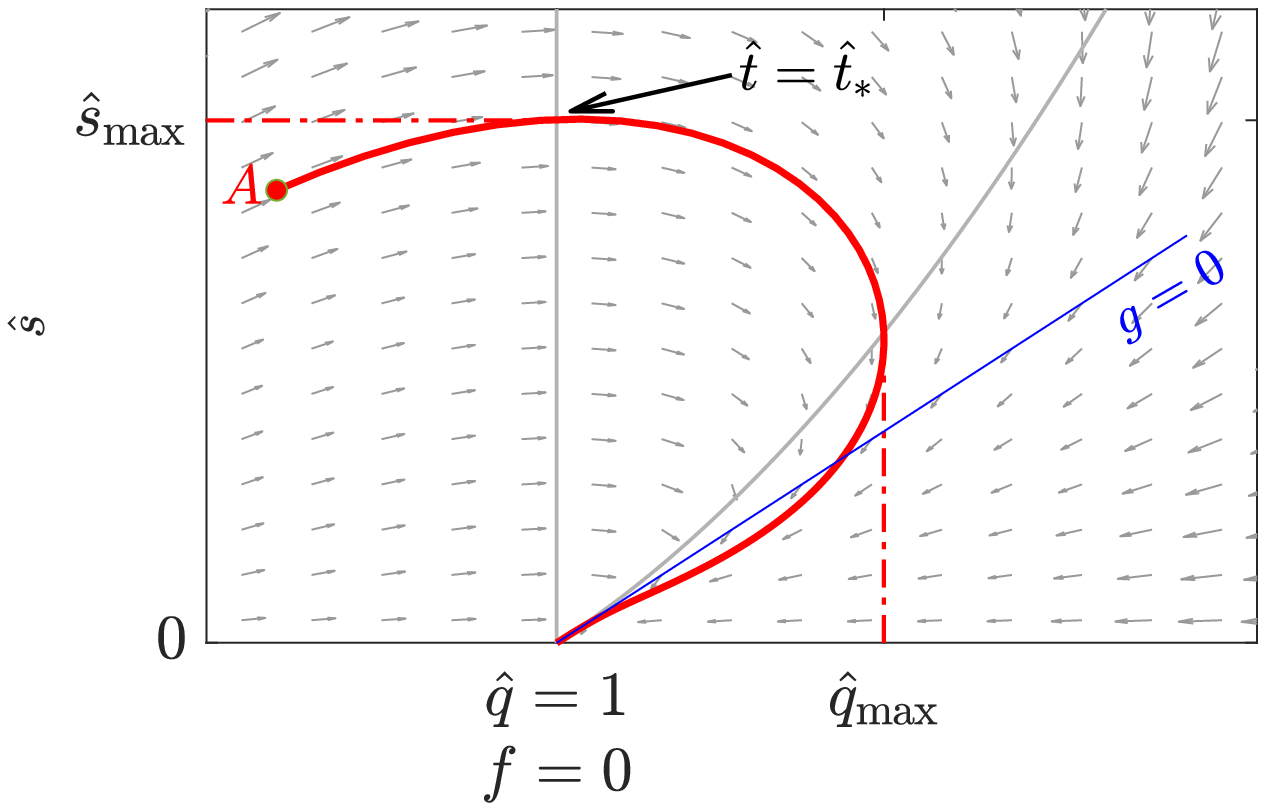}
  \caption{Illustration of the the phase plane of $(\hat{q},\hat{s})$.
    Left figure: $n\geq3$, right figure $n=2$.}\label{fig:EP_sqhat}
\end{figure}

Now, we aim to obtain a better decay estimate on $s(t)$.

For $n\geq3$. we observe from \eqref{eq:qshat} and Figure \ref{fig:EP_sqhat} that
$\hat{s}$ obtain its maximum value $\hat{s}_{\max}$ at a finite time $\hat{t}_*$ when
$\hat{q}(\hat{t}_*)=\frac2n$ (or 
$\hat{q}>\frac2n$ if $\hat{q}_0>\frac2n$, where $\hat{t}_*=0$).
We can write
\begin{equation}\label{eq:shat}
  \hat{s}(\hat{t})=\hat{s}_{\max}\exp\left[\int_{\hat{t}_*}^{\hat{t}}
    \left(2-n\hat{q}(\hat{\tau})\right)\,d\hat{\tau}\right]
=\hat{s}_{\max}\exp\left[(2-n)(\hat{t}-\hat{t}_*)+n\int_{\hat{t}^*}^{\hat{t}}
  \left(1-\hat{q}(\hat{\tau})\right)\,d\hat{\tau}\right],
\end{equation}
where we define $f(\hat{t})=q(\hat{t})-1$ which satisfies
\[f'=-f^2-f+\kappa\hat{s}\geq -f^2+f,\quad f(t_*)\geq \frac2n-1.\]
Explicit calculation yields
\[f(\hat{t})\geq-\frac{n-2}{n-2+2e^{\hat{t}-\hat{t}_*}}\geq-\frac{n-2}{2}e^{-(\hat{t}-\hat{t}_*)},
  \quad\forall~\hat{t}\geq\hat{t}_*.\]
Therefore, the last integral in \eqref{eq:shat} is uniformly bounded
\[\int_{\hat{t}_*}^{\hat{t}}
  \left(1-\hat{q}(\hat{\tau})\right)\,d\hat{\tau}\leq\int_0^\infty\frac{n-2}{2}e^{-\hat{\tau}}\,d\hat{\tau}=\frac{n-2}{2}.\]
Finally, we obtain
\begin{equation}\label{eq:stimproved}
  s(t)=\hat{s}(\hat{t})(t+1)^{-2}\leq\hat{s}_{\max}e^{(n-2)\hat{t}_*+\frac{n(n-2)}{2}}(t+1)^{2-n}(t+1)^{-2}=:C_s(t+1)^{-n}.
\end{equation}

We are left with the case $n=2$, which turns out to be critical. To
obtain the logarithmic improvement in \eqref{eq:qsasymp}, we need to
show $\hat{s}(\hat{t})\lesssim (\hat{t}+1)^{-1}$. Define two new variables
\[f=\hat{q}-1, \quad g=\kappa\hat{s}-\hat{q}+1.\]
Then, \eqref{eq:qshat} can be equivalently expressed as
\[
  \begin{cases}
    f'=-f^2+g,\\
   g'=(-1-2f)g-f^2,
 \end{cases}
 \quad\text{with}\quad
  \begin{cases}
    f(\hat{t}_*)=\hat{q}(\hat{t}_*)-1\geq0,\\
    g(\hat{t}_*)=\kappa\hat{s}(\hat{t}_*)-f(\hat{t}_*).
 \end{cases}
  \]
Clearly, $f\geq0$ is an invariant region. Then, we can easily find an upper
bound of $g$ by
\[g(\hat{t})\leq \min\left\{g(\hat{t}_*) e^{-(\hat{t}-\hat{t}_*)},0\right\}.\]
Plugging back into the dynamics of $f$, we immediately obtain an upper
bound of $f$
\[f(\hat{t})\leq C(\hat{t}-\hat{t}_*+1)^{-1},
  \quad\forall~\hat{t}\geq\hat{t}_*,\]
where $C$ depends on $f(\hat{t}_*)$ and $g(\hat{t}_*)$.
Finally, we conclude that
\[\hat{s}(\hat{t})=\frac{1}{\kappa}\big(f(\hat{t})+g(\hat{t})\big)\lesssim(\hat{t}+1)^{-1}.\]
\end{proof}

\begin{remark}\label{rmk:EP-c0-conv}
In \eqref{eq:qsasymp}, the constant $C_q\geq1$ as
$\hat{q}_{\max}\geq1$. Note that if $s_0$ is small enough,
then $C_q$ is close to 1. In particular, as illustrated in the left
figure in Figure \ref{fig:EP_sqhat}, For the trajectory start at point
$B$, $C_q=1$.
\end{remark}

\subsubsection{Explicit subcritical regions}
We now switch to discuss the dynamics of $(p,\rho)$ in \eqref{eq:EPfull}.
Recall
\begin{equation}\label{eq:prho}
  \begin{cases}
    p'=-p^2+\kappa(\rho-(n-1)s),\\
    \rho'=-\rho(p+(n-1)q).
  \end{cases}
\end{equation}

Note that even if $(q,s)$ stays bounded uniformly in time, they affect
the dynamics of $(p,\rho)$ when $n\geq2$, and hence the subcritical
region $\Sigma$ could be different than \eqref{eq:Sigma1D}.

The goal here is to find a more explicit subcritical region. In
particular, we need to make sure that the subcritical region is not an
empty set in general.

To better understanding the dynamics of \eqref{eq:prho}, we proceed
with the following transformations. First, consider the dynamics of
$(p/\rho, 1/\rho)$
\begin{equation}\label{eq:prho2}
  \begin{cases}
    \displaystyle\left(\frac{p}{\rho}\right)'=\,\kappa
    +(n-1)\left(q\cdot\frac{p}{\rho}-\kappa s\cdot\frac{1}{\rho}\right),\\
    \displaystyle\left(\frac{1}{\rho}\right)'=\,\frac{p}{\rho}+(n-1)\frac{q}{\rho}.
  \end{cases}
\end{equation}
To absorb the explicit dependence on $q$, we introduce new quantities
$(w,v)$ along the characteristic paths as follows
\begin{equation}\label{eq:wvdef}
  w = \frac{p}{\rho}\cdot e^{(n-1)A(t)},\quad
  v = \frac{1}{\rho}\cdot e^{(n-1)A(t)},
\end{equation}
where $A(t)$ is defined as
\begin{equation}\label{eq:At}
 A(t):=-\int_0^tq(\tau)\,d\tau=\frac{1}{n}\ln\frac{s(t)}{s_0}.
\end{equation}
The second equality directly comes from \eqref{eq:st}. As we have
already know $s$ is uniformly bounded in time, so does $A$.
Then, the dynamics of $(w,v)$ reads
\begin{equation}\label{eq:wv}
  \begin{cases}
    \displaystyle w'=\,\kappa e^{(n-1)A}-\kappa\big(c+(n-1)s\big)v,\\
    v'=\,w.
  \end{cases}
\end{equation}

Let us first summarize the threshold condition for $n=1$. In this
case, \eqref{eq:wv} simply becomes
\[w'=\kappa, \quad v'=w.\]
Therefore, we obtain
\[w(t)=w_0+\kappa t,\quad\text{and}\quad v(t)=v_0+w_0t+\frac{\kappa}{2}t^2.\]
Then, $v(t)$ won't reach zero if and only if $w_0>-\sqrt{2\kappa
  v_0}$. From the definition of $(w,v)$ \eqref{eq:wvdef}, this is
equivalent to $p_0>-\sqrt{2\kappa\rho_0}$.

When $n\geq2$, we have $e^{(n-1)A}\not\equiv1$ and
$(n-1)sv\not\equiv0$. The two terms reflect the contribution of
$(q,s)$ to the dynamics of $(w,v)$.
In particular,
$e^{(n-1)A(t)}=\left(\frac{s(t)}{s_0}\right)^{\frac{n-1}{n}}$ vanishes
as $t\to\infty$ due to Lemma \ref{lem:qsasymp}. Therefore, the
behavior of $(w,v)$ is different from the 1D case.

Let us state our result.
\begin{theorem}\label{thm:EP-prho-sub}
  Let $n\geq2$. There exists 
  threshold function $\sigma_+~:~\R_+\to\R_+$, depending on $(q_0,s_0)$, such that
  \[\left\{(p_0, q_0, s_0,
    \rho_0)~\left|~p_0>-\rho_0\sigma_+\left(\frac{1}{\rho_0};~q_0,s_0\right)\right.\right\}\subset \Sigma.\]
\end{theorem}
\begin{remark}
  For $n=1$, $\sigma_+(x)=\sqrt{2\kappa x}$. For $n\geq2$, we obtain a
  similar condition, allowing $p_0$ to be negative. $\sigma_+$ will
  depend on $(q_0,s_0)$, indicating the effect of the spectral gap.
\end{remark}

Let us first consider the case $n\geq3$. Write
\begin{equation}\label{eq:wdyn}
  w(t) = w_0+\kappa \int_0^t e^{(n-1)A(\tau)}\,d\tau-\kappa
  (n-1)\int_0^ts(\tau)v(\tau)\,d\tau.
\end{equation}

\textit{Step 1: upper bounds on $w$ and $v$.} Apply Lemma
\ref{lem:qsasymp} and get
\begin{align*}
  w(t)\leq &\,w_0+\kappa
  s_0^{-\frac{n-1}{n}}\int_0^ts(\tau)^{\frac{n-1}{n}}\,d\tau\,
  \leq w_0+\kappa\left(\frac{C_s}{s_0}\right)^{\frac{n-1}{n}}\int_0^t(\tau+1)^{-(n-1)}\,d\tau\\
  \leq&\,w_0+\frac{\kappa}{n-2}\left(\frac{C_s}{s_0}\right)^{\frac{n-1}{n}}=:w_0+C(q_0,s_0),
\end{align*}
for $n\geq 3$. Therefore, unlike 1D where $w$ can grow linearly in
time, $w$ is uniformly bounded. And $v$ can grow at most linearly
\begin{equation}\label{eq:vupper}
  v(t)\leq
  v_0+\left(w_0+C(q_0,s_0)\right)t.
\end{equation}
\begin{remark}
  If
  $w_0<-C(q_0,s_0)$,
  or equivalently
  $p_0<-C(q_0,s_0)\rho_0$,
  $v(t)$ will become negative in finite time. Hence, such initial data
  lie in the supercritical region.
\end{remark}

\textit{Step 2: lower bounds on $w$ and $v$, assuming
$w_0>-C(q_0,s_0)$.}
Let us control the two integrals in \eqref{eq:wdyn} one by one.
For the first term, by \eqref{eq:qsasymp} and \eqref{eq:At},
we have
\[A(t)\geq-\int_0^tC_q(\tau+1)^{-1}\,d\tau=-C_q\ln(t+1).\]
Then,
\[\kappa\int_0^t
  e^{(n-1)A(\tau)}\,d\tau\geq\kappa\int_0^t(\tau+1)^{-C_q(n-1)}\,d\tau=
  \frac{\kappa}{\gamma+1}\left(1-(t+1)^{-\gamma-1}\right),\]
where $\gamma=C_q(n-1)-2$. Note that from Remark \ref{rmk:EP-c0-conv},
$C_q\geq1$ (strict inequality for $n=3$) , we have $\gamma>0$ for $n\geq3$.

For the second term, apply \eqref{eq:qsasymp} and \eqref{eq:vupper}
\begin{align*}
  \int_0^ts(\tau)v(\tau)\,d\tau\leq&\,\int_0^t C_s(\tau+1)^{-n}\big(
                                     v_0+\left(w_0+C(q_0,s_0)\right)\tau\big)\,d\tau\\
  \leq&\,C_s\left(\frac{v_0}{n-1}+\frac{w_0+C(q_0,s_0)}{n-2}\right).
\end{align*}

Put the two estimates together, we have
\[w(t)\geq
  w_0-C_s\left(\frac{v_0}{n-1}+\frac{w_0+C(q_0,s_0)}{n-2}\right)+\frac{\kappa}{\gamma+1}\left(1-(t+1)^{-\gamma-1}\right).\]
Denote
\[D:=-w_0+C_s\left(\frac{v_0}{n-1}+\frac{w_0+C(q_0,s_0)}{n-2}\right).\]
Then, we get
\begin{align*}
  w(t)\geq&\,
  -D+\frac{\kappa}{\gamma+1}\left(1-(t+1)^{-\gamma-1}\right),\\
  v(t)\geq&\, v_0+\left(-D+\frac{\kappa}{\gamma+1}\right)t
            -\frac{\kappa}{\gamma(\gamma+1)}\left(1-(t+1)^{-\gamma}\right).
\end{align*}

To complete the lower bound estimate, we state the following lemma.
\begin{lemma}\label{lem:lowerbound}
  Let $y(t)$ be a function defined as
  \[y(t)= v_0+\left(\frac{\kappa}{\gamma+1}-D\right)t
    -\frac{\kappa}{\gamma(\gamma+1)}\left(1-(t+1)^{-\gamma}\right).\]
  Then, there exists a constant $D_{crit}=D_{crit}(v_0)>0$, depending
  on initial data $v_0$, and parameters
  $\gamma, \kappa$, such that if $D< D_{crit}$, then $y(t)>0$ for
  all $t\in[0,\infty)$.
\end{lemma}
\begin{proof}
  For $D\leq0$, the result is trivial as $y(t)\geq v_0>0$.
  On the other hand, if $D>\frac{\kappa}{\gamma+1}$, $y(t)\leq
  v_0-\left(D-\frac{\kappa}{\gamma+1}\right)t$ will reach zero in
  finite time, regardless of the choice of $v_0$. Hence, $D_{crit}\leq
  \frac{\kappa}{\gamma+1}$.

  Let us focus on $D\in(0, \frac{\kappa}{\gamma+1}]$. For simplified
  notations, let $z=1-\frac{\gamma+1}{\kappa}D\in[0,1)$.
  The minimum of $y$ is attained at
  $t_*=z^{-\frac{1}{\gamma+1}}-1$.
  We calculate
  \begin{align*}
    y_{\min}=y(t_*)=&\,v_0+\frac{\kappa}{\gamma+1}z^{\frac{\gamma}{\gamma+1}}
    -\frac{\kappa}{\gamma+1}z-\frac{\kappa}{\gamma(\gamma+1)}
                                  \left(1-z^{\frac{\gamma}{\gamma+1}}\right)\\
    =&\,v_0+\frac{\kappa}{\gamma}z^{\frac{\gamma}{\gamma+1}}
       -\frac{\kappa}{\gamma+1}\left(z+\frac{1}{\gamma}\right)=:F(z).
  \end{align*}
  We can view the minimum as a function of $z$.
  Observe that $F$ is an increasing function in $[0,1]$, $F(1)=v_0>0$,
  and $F(0)=v_0-\frac{\kappa}{\gamma(\gamma+1)}$. Therefore, we have
  \begin{itemize}
    \item If $v_0<\frac{\kappa}{\gamma(\gamma+1)}$, $F$ has a unique
      root $z_*\in(0,1)$, and $F(z)>0$ for all $z>z_*$. Therefore,
      if $D<D_{crit}:=\frac{\kappa}{\gamma+1}(1-z_*)$, then
      $y_{\min}>0$.
    \item If $v_0\geq \frac{\kappa}{\gamma(\gamma+1)}$, $F(z)>0$ for
      all $z\in(0,1)$. Hence, if
      $D<D_{crit}:=\frac{\kappa}{\gamma+1}$, $y_{\min}>0$.
  \end{itemize}
\end{proof}

\textit{Step 3: conclusion}. As a direct consequence of Lemma
\ref{lem:lowerbound}, we obtain a
subcritical condition $D<D_{crit}$, which can be conveniently
rewritten as
\begin{equation}\label{eq:EP_c0_sub}
w_0>\left[-D_{crit}(v_0)+C_s\left(\frac{v_0}{n-1}+\frac{C(q_0,s_0)}{n-2}\right)\right]\left(1-\frac{C_s}{n-2}\right)^{-1}=:-\sigma_+(v_0).
\end{equation}
This finishes the proof of Theorem \ref{thm:EP-prho-sub}.

\begin{remark}
  The constant $C_s$ can be small so that the right hand side of
  \eqref{eq:EP_c0_sub} is negative. Indeed, in the case when
  $q_0>\frac{2}{n}$, we have $\hat{s}_{\max}=s_0$ and $\hat{t}_*=0$ in
  \eqref{eq:stimproved}. Then, $C_s=s_0e^{\frac{n(n-2)}{2}}$ is small
  as long as $s_0$ is small.
  For general case, particularly $q_0<0$, a similar argument works for
  the dynamics starting at time $t=t_*$. Since $t_*$ is finite, it is
  easy to control $v(t)$ for $t<t_*$. We omit the technical details
  here for simplicity.
\end{remark}

For $n=2$, due to its criticality, the
calculation would be slightly different. Thanks to the logarithmic
improvement in \eqref{eq:sasymp}, we are able to obtain a similar
result. We shall only sketch the proof, highlighting the
difference.

First, $w(t)$ is not bounded by a constant, but could have a
logarithmic growth.
\begin{align*}
  w(t)\leq &~w_0+2\kappa\left(\frac{C_s}{s_0}\right)^{\frac12}\Big(\ln(t+1)+1\Big)^{\frac12},\\
  v(t) \leq &~v_0+w_0t+ 2\kappa\left(\frac{C_s}{s_0}\right)^{\frac12}t\Big(\ln(t+1)+1\Big)^{\frac12}.
\end{align*}
Next, for the lower bound, since the estimates above do not imply
boundedness of $\int_0^ts(\tau)v(\tau)\,d\tau$, the previous estimates
for $n\geq3$ does not follow. Instead, we write
\[w'=\kappa\left(\frac{s}{s_0}\right)^{\frac12}-\kappa sv=\kappa
  s^{\frac12}\left(s_0^{-\frac12}-s^{\frac12}v\right),\]
and the term $s^{\frac12}v$ is bounded and
\begin{align*}
 s^{\frac12}v\leq&~ C_s^{\frac12}(t+1)^{-1}\Big(\ln(t+1)+1\Big)^{-\frac12}
 \left(v_0+w_0t+2\kappa\left(\frac{C_s}{s_0}\right)^{\frac12}t\Big(\ln(t+1)+1\Big)^{\frac12}\right)\\
 \rightarrow&~2\kappa C_ss_0^{-\frac12},\quad\text{as}\quad t\to+\infty.
\end{align*}
Therefore, if we choose a small $C_s$ such that $2\kappa C_s<1$, then
$w'$ will eventually become positive. One can continue with a similar
argument as in the $n\geq3$ case to obtain a threshold condition in
Theorem \ref{thm:EP-prho-sub}. The technical details will be omitted.

\subsection{Multi-dimensional case with positive constant background}
Now, we study the Euler-Poisson equations with constant background $c>0$. It is
known that the behavior of the solution is very different from the
zero background case. We will start with analyzing the $(q,s)$ pair in
the phase plane.

\subsubsection{Uniform boundedness of $(q,s)$}
Recall the $(q,s)$ dynamics for the case $c>0$
\begin{equation}\label{eq:qsc}
  \begin{cases}
    q'=-q^2+\kappa s,\\
    s'=-(ns+c)q.
  \end{cases}
\end{equation}

Compared with the zero background case, the main difference is that,
since $s$ can be negative, $q(t)\geq0$ is no longer an invariant
region. So Lemma~\ref{lem:qsplus} does not apply.
In the phase plane of $(q,s)$, The steady state $(0,0)$ is not
an attractor. As illustrated in Figure~\ref{fig:EP_sq_cp}, the
trajectories of $(q,s)$, if bounded, form periodic orbits, and do not converge as
the time approaches infinity.

\begin{figure}[ht]
  \includegraphics[width=.7\linewidth]{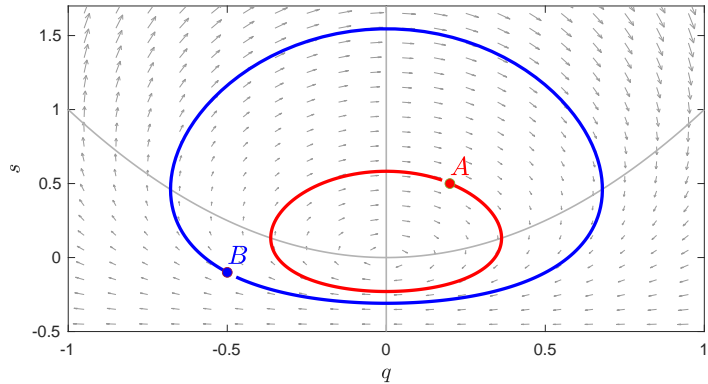}
  \caption{Illustration of the the phase plane of $(q,s)$ with $c=1$ and $N=2$.
    For any initial data (both $A$ and $B$ as examples), the solution
    are bounded uniformly in
    time. The trajectories form close orbits around $(0,0)$ that are symmetric in
    $s$-axis. The solutions are periodic in time.}\label{fig:EP_sq_cp}
\end{figure}

\begin{theorem}[Boundedness of $(q,s)$]\label{thm:EPb2}
  Let $n\geq2$. Consider the $(q,s)$ dynamics in \eqref{eq:qsc} with bounded
  initial conditions $(q_0,s_0)$ such that $s_0>-\frac{c}{n}$.
  Then, $(q(t), s(t))$ remains bounded in all time.
  Moreover,
  the trajectory of $(q(t),s(t))$ stays on a bounded periodic orbit
  in the $(q,s)$-plane. 
\end{theorem}

\def\st{\tilde{s}}

\begin{proof}
Let us perform the following convenient transformation
\[\st=s+\frac{c}{n}.\]
The dynamics of $(q,\tilde{s})$ reads
\begin{equation}\label{eq:qst}
  \begin{cases}
    q'=-q^2+\kappa \st-\frac{\kappa c}{n},\\
    \st'=-n\st q,
  \end{cases}
\end{equation}
and we are only interested in the case when $\st_0>0$, which
clearly preserves in time.

  We first express $\st$ in terms of $q$ as
  \[\st(t)=\st_0\exp\left[-n\int_0^tq(\tau)d\tau\right].\]
Immediately, we obtain a lower bound $\st(t)>0$ as long as $q$ stays
bounded.

Assume by contradiction, there exist a first time $T_*$ such that solution
becomes unbounded. $T_*$ can be either finite (corresponding to finite time
blowup), or infinity. Then, at least one of the three scenarios
happen: 
\[\lim_{t\to T_*-}q(t)=+\infty,\quad
  \lim_{t\to T_*-}q(t)=-\infty,\quad\text{or}\quad
  \lim_{t\to T_*-}\st(t)=+\infty.\]
We will show all three scenarios leads to contradictions.

First, if $\lim_{t\to T_*-}q(t)=+\infty$, there must exists time $t_0\in[0,T_*)$
such that $q(t)>0$ for every $t\in[t_0,T_*)$.
Then, $\st'(t)<0$ and hence $\st(t)\leq\st(t_0)$ for every $t\in[t_0,T_*)$.
On the other hand, from the dynamics of $q$, we have
\[q'(t)\leq-q^2(t)+\kappa\st(t_0)<0,\quad
\text{if}\quad
q(t)\geq\sqrt{\kappa\st(t_0)},\quad\forall~t\in[t_0,T_*).\]
This implies that $q(t)\leq\max\{q(t_0),\,\sqrt{\kappa\st(t_0)}\}$,
  which leads to a contradiction.

Second, if $\lim_{t\to T_*-}q(t)=-\infty$, there must exists a time $t_0\in[0,T_*)$
such that $q(t)<-\sqrt\frac{\kappa c}{n}$ for every $t\in[t_0,T_*)$.
A rough estimate on $q$ would read
\[q'\geq-q^2-\frac{\kappa c}{n}\geq-2q^2,\quad\forall t\in[t_0,T_*),\]
The rest of the proof will be identical to Theorem~\ref{thm:EPb1},
with only changes on the constant coefficients, as well as a shift of
time variable by $t_0$. The result shows that
$q(t)$ has a lower bound in all time, which clearly leads to a
contradiction.

Third, $\lim_{t\to T_*-}\st(t)=+\infty$, there must exists a time
$t_0\in[0,T_*)$ such that $\st(t)>\kappa^{-1}(Q^2+\frac{\kappa c}{n}
+1)$ for every $t\in[t_0,T_*)$, where
$Q=\sup_{t\in[0,T_*)}|q(t)|$ which is finite. Then,
\[q'(t)=-q(t)^2+\kappa\st(t)-\frac{\kappa c}{n}\geq
  -Q^2+\kappa\st(t)-\frac{\kappa c}{n}>1,\quad\forall~t\in[t_o,T_*).\]
Then, $q$ has a better lower bound
\[q(t)\geq -Q+(t-t_0)\geq\begin{cases}-Q&t_0\leq
    t<t_0+Q\\0&t>t_0+Q\end{cases}\]
and therefore
\[\st(t)=\st(t_0)\exp\left[-n\int_{t_0}^tq(\tau)d\tau\right]\leq
\st(t_0)e^{nQ^2},\quad\forall~t\in[t_0,T_*).\]
This leads to a contradiction.

Finally, the trajectory of the dynamics is symmetric in $q$ (by simple
observations from \eqref{eq:qst}). Therefore, the solution is periodic
in time, and travels along a closed orbit in the $(q,\st)$-phase plane.
\end{proof}

\subsubsection{Subcritical regions}
The dynamics of $(p,\rho)$ in \eqref{eq:EPfull} reads
\begin{equation}\label{eq:prho}
  \begin{cases}
    p'=-p^2+\kappa(\rho-c-(n-1)s),\\
    \rho'=-\rho(p+(n-1)q).
  \end{cases}
\end{equation}
We introduce the same variables $(w,v)$ in \eqref{eq:wvdef},
with $A(t)$ defined as
\[A(t):=-\int_0^tq(\tau)\,d\tau=\frac{1}{n}\ln\frac{\st(t)}{\st_0}.\]
The dynamics of $(w,v)$ has the form
\begin{equation}\label{eq:wvc}
  \begin{cases}
    \displaystyle w'=\,\kappa e^{(n-1)A}-\kappa\big(c+(n-1)s\big)v,\\
    v'=\,w.
  \end{cases}
\end{equation}

If $n=1$, the dynamics \eqref{eq:wvc} is simply a closed linear system
\begin{equation}\label{eq:1Dwv}
  w' = \kappa(1-cv),\quad v'=w,
\end{equation}
which can be solved explicitly.
The trajectory of the solution $(w,v)$ in the phase plane form a ellipse
\[\frac{w^2}{c\kappa}+\left(v-\frac1c\right)^2=R^2,\]
where $R$ is determined by the initial condition $(w_0,v_0)$.
$v(t)>0$ is then equivalent to $R<\frac1c$, or 
$w_0^2<\kappa(2v_0-cv_0^2)$. This leads to the sharp threshold
condition in \eqref{eq:Sigma1D}.

However, with the effect of the spectral gap, it is very difficult to
extend the 1D result in the multiple dimensions.
Unlike the zero background case where the solution intends to converge
to the equilibrium, the comparison principle fails as the solution
oscillates. Moreover, the period for $(q,s)$ does not necessarily
match with the period in \eqref{eq:1Dwv}, leading to a more chaotic
dynamics. Hence, an explicit expression of the subcritical regions,
like Theorem \ref{thm:EP-prho-sub}, remains to be a challenging
open problem.

\section{Application to the Euler-alignment equations}\label{sec:EA}
In this section, we discuss the Euler-alignment equations
\begin{align*}
  &\pa_t\rho+\div(\rho\u)=0,\\
  &\pa_t\u+(\u\cdot\grad)\u=\int_{\R^n}\phi(|\x-\y|)(\u(\y)-\u(\x))\rho(\y)d\y.
\end{align*}
The system arises as the macroscopic representation of the
Cucker-Smale flocking dynamics, describing the emergent phenomenon of
animal flocks.

The nonlocal alignment force is modeled through an \emph{influence
  function} $\phi$. Here, we assume $\phi$ is bounded, Lipschitz,
non-increasing, and decays slowly at infinity
\begin{equation}\label{EA:decay}
  \int^\infty\phi(r)\,dr=\infty.
\end{equation}

We state the local wellposedness theorem for the Euler-alignment
equations, the same as Theorem \ref{thm:EPlocal}.
The proof can be found in, for instance, \cite{tadmor2014critical,tan2020euler}.
\begin{theorem}[Local wellposedness]\label{thm:EAlocal}
  Consider the Euler-alignment equations with initial data
  $\rho_0\in H^s(\R^n)$ and
  $\u_0\in H^{s+1}(\R^n) ^n$, for $s>\frac{n}{2}$.
  Then, there exists a time $T>0$ such that
  the solution
  \begin{equation}\label{eq:EAregularity}
    (\rho,\u)\in C([0,T],  H^{s}(\R^n))\times
    C([0,T], H^{s+1}(\R^n))^n.
  \end{equation}
  Moreover, the life span $T$ can be extended as long as
  \begin{equation}\label{eq:EABKM}
    \int_0^T\|\grad\u(\cdot,t)\|_{L^\infty}\,dt<+\infty.
  \end{equation}
\end{theorem}

The slow-decay condition \eqref{EA:decay} is known to ensure the asymptotic flocking
behavior.
\begin{theorem}[Strong solution must flock \cite{tadmor2014critical}]\label{thm:flock}
  Let $(\rho, \u)$ be a strong solution of the Euler-alignment system,
  with compactly supported initial density $\rho_0$, and the influence
  function satisfies $\phi$ the condition \eqref{EA:decay}. Then, the
  solution must flock, namely, there exists a constant $D$, depending
  on the initial data, such that
  \begin{equation}\label{eq:flocking}
    \text{supp}(\rho(\cdot,t))\subset B_D(0),\quad\forall~t\geq0,
  \end{equation}
  where $B_D(0)$ is the ball in $\R^n$ that is centered at origin with
  radius $D$.
  Moreover, the solution exhibits fast alignment,
  \begin{equation}\label{eq:fastalign}
    V(t) \leq V_0~e^{-\nu t},\quad V(t):=\sup_{x,y}|u(x,t)-u(y,t)|.
  \end{equation}
  with an exponential rate of decay
  \begin{equation}\label{eq:nu}
    \nu=\phi(2D)\|\rho_0\|_{L^1}>0.
  \end{equation}
\end{theorem}

In the following, we focus on the radially symmetric setup
\eqref{eq:radial}. The starting point is to verify that the force $\F$
has the form \eqref{eq:forceradial}, so radial symmetry preserves in time.
Express the nonlocal alignment force as
\[\F=\int_{\R^n}\phi(|\x-\y|)(\u(\y,t)-\u(\x,t))\rho(\y,t)d\y=
  \mathcal{L}(\rho\u)-\u\mathcal{L}\rho,\]
where
\[\mathcal{L}f(\x):=\int_{\R^n}\phi(|\x-\y|)f(\y)d\y.\]

Under radial symmetry \eqref{eq:radial}, it is easy to check that
$\mathcal{L}\rho$ is a radial function, as the convolution of radial
functions are radial. Let us denote
\begin{equation}\label{eq:psi}
  \psi(r)=\mathcal{L}\rho.
\end{equation}

The term $\mathcal{L}(\rho\u)$ can be expressed as follows.
\begin{proposition}\label{prop:zetaradial}
  The vector-valued function $\mathcal{L}(\rho\u)$ can be written as
  \[\frac{\x}{r}\zeta(r)=\mathcal{L}(\rho\u),\]
  where $\zeta$ is defined as
\begin{equation}\label{eq:zeta}
 \zeta(r)=\int_{\R^n}\phi(|r\e_1-\z|)\rho(|\z|)\frac{z_1}{|\z|}u(|\z|)d\z,
\end{equation}
 with $\e_1=[1,0,\cdots,0]^T$.
\end{proposition}
\begin{proof}
 Let $U$ be a unitary matrix in $\R^n$ such that its first column is
 $\x/r$, namely
 \[\x=rU{\bf e}_1.\]
 
Since the length $|\cdot|$ is invariant under unitary
transformation, we have
\[|\x-\y|=|U^T(\x-\y)|=|r\e_1-U^Ty|.\]
Then, we can compute
\begin{align*}
  \mathcal{L}(\rho\u)=&~                        
  \int_{\R^n}\phi(|r\e_1-U^T\y|)\rho(|\y|)\frac{\y}{|\y|}u(|\y|)d\y
  =  \int_{\R^n}\phi(|r\e_1-\z|)\rho(|\z|)\frac{U\z}{|\z|}u(|\z|)d\z\\
  =&~
     \sum_{k=1}^nU\e_k\int_{\R^n}\phi(|r\e_1-\z|)\rho(|\z|)\frac{z_k}{|\z|}u(|\z|)d\z\\
  =&~\frac{\x}{r}\int_{\R^n}\phi(|r\e_1-\z|)\rho(|\z|)\frac{z_1}{|\z|}u(|\z|)d\z.
\end{align*}
For the last equality, we use the fact that for $k\geq2$, the function
is odd with respect to $z_k$, and hence the integral is zero.
\end{proof}

Combining Proposition \ref{prop:zetaradial} and \eqref{eq:psi}, we
have verified \eqref{eq:forceradial} with $F=\zeta-\psi u$.
The dynamics of the radial profile $(\rho,u)$ reads
\[\begin{cases}
  \rho_t+(\rho u)_r=-\displaystyle (n-1)\frac{\rho u}{r},\\
  u_t+uu_r=\zeta-\psi u.
\end{cases}
\]

Let us write out the dynamics of the pair $(p,q)$ in \eqref{eq:pair} as follows
\[
  \begin{cases}
    p'=-p^2+\zeta_r-p\psi-u\psi_r,\\
    q'=-q^2+\displaystyle\frac{\zeta}{r}-q\psi,
  \end{cases}
\]
where again $'=\pa_t+u\pa_r$ denotes the material derivative.

To eliminate the nonlocal term $\zeta_r$, we follow the idea introduced in
\cite{carrillo2016critical}. Calculate the dynamics of $\psi$
\[\psi_t=\pa_t\mathcal{L}\rho=-\div\mathcal{L}(\rho\u)=-\zeta_r-(n-1)\frac{\zeta}{r}.\]
Then, adding the dynamics of $p$ and $\psi$ would yield
\begin{equation}\label{eq:GdynGen}
  (p+\psi)'=-p(p+\psi)-(n-1)\frac{\zeta}{r}.
\end{equation}

Let $G=p+\psi$. We summarize the dynamics on $(\rho, G)$
\begin{equation}\label{eq:EA-rhoG}
  \begin{cases}
   \rho_t+(\rho u)_r=- (n-1)\rho q,\\
    G_t+(Gu)_r=-(n-1)\displaystyle\frac{\zeta}{r}.
  \end{cases}
\end{equation}

\subsection{The one-dimensional case}
When $n=1$, the right hand side of \eqref{eq:EA-rhoG} vanishes.
In particular, $G$ satisfies the continuity equation $G_t+(Gu)_r=0$.
Therefore, $G\geq0$ is an invariant region. Further investigation
leads to a sharp threshold condition.
\begin{theorem}[1D sharp threshold \cite{carrillo2016critical}]\label{thm:EA-1D}
  Consider the Euler-alignment system in 1D.
  \begin{itemize}
  \item (Subcritical region) If $\inf G_0\geq0$, the solution is globally regular.
  \item (Supercritical region) If $\inf G_0<0$, there exists a finite time blowup.
  \end{itemize}
\end{theorem}

\subsection{The effect of the spectral gap}
When $n\geq2$, extra terms appear in \eqref{eq:EA-rhoG}  involving
$q$ and $\zeta/r$,
which can not be locally expressed in terms of $(\rho, G)$ along a
characteristic path.
These two quantities encode the main difference between 1D and
multi-dimensions, and hence is related to the spectral gap effect.

Let us first focus on $\zeta/r$.

One way to eliminate the term $\zeta/r$ is to take a linear combination of
$G=p+\psi$ and $q$ as follows
\[(d+\psi)'=(p+\psi+(n-1)q)'=-p(p+\psi)-(n-1)q(q+\psi),\]
where $d=\div\u$.
However, this does not reduce the problem to the
one-dimensional case, as the right hand side of the dynamics is different from
$-d(d+\psi)$. One needs to control the spectral gap $\eta$ in
\eqref{eq:spectralgap}, which could be difficult.
This approach has been investigated in \cite{he2017global} only for $n=2$.

As we have argued throughout the paper, we shall study the pair
$(p,q)$ instead of $d$.
To this end, we obtain a bound on $\zeta/r$.

\begin{proposition}[Boundedness of $\zeta/r$]\label{prop:zeta}
  The quantity
  $\frac{\zeta(r,t)}{r}$ is uniformly bounded in
  $(r,t)\in\R_+\times\R_+$. Moreover, it decays exponentially in time,
  with the same rate as in \eqref{eq:fastalign}, thus, there exists a
  constant $C_0$, depending on the initial data, such that
  \begin{equation}\label{eq:zetabound}
    \sup_{r>0}\frac{|\zeta(r,t)|}{r}\leq B(t):=C_0e^{-\nu t}.
  \end{equation}
\end{proposition}
\begin{proof}
  We estimate $\zeta$ from its definition \eqref{eq:zeta}.
\begin{align*}
  |\zeta(r,t)|=&\,\left|\int_{\R^n}(\phi(|r\e_1-\z|)-\phi(|\z|))\rho(|\z|,t)\frac{z_1}{|\z|}u(|\z|,t)\,d\z\right|\\
  \leq&\,\int_{\R^n}\|\phi'\|_{L^\infty}|r\e_1|\rho(|\z|,t)\left|\frac{z_1}{|\z|}\right|u(|\z|,t)\,d\z\\
  \leq&\,
        r\|\phi'\|_{L^\infty}\|\rho(\cdot,t)\|_{L^1}\|u(\cdot,t)\|_{L^\infty}
        \leq r\|\phi'\|_{L^\infty}\|\rho_0\|_{L^1}\|u_0\|_{L^\infty}e^{-\nu t}.
\end{align*}
For the first equality, odd symmetry in $z_1$ is used. For the last
inequality, the fast alignment estimate \eqref{eq:fastalign} is
applied. Note that due to the symmetry on $\u$, it is easy to check
that $V(t)=2\|u(\cdot,t)\|_{L^\infty}$.

This ends the proof of \eqref{eq:zetabound}, with $C_0=\|\phi'\|_{L^\infty}\|\rho_0\|_{L^1}\|u_0\|_{L^\infty}$.
\end{proof}

\begin{remark}
  While $\zeta/r$ is bounded and decay in time, it does not
  necessarily has a definite sign. Therefore, $G\geq0$ is no longer an
  invariant region, and we do not expect that the sharp threshold result in
  1D (Theorem \ref{thm:EA-1D}) remains true in multi-dimensions.
\end{remark}

Next, we work on $q$. Recall its dynamics
\begin{equation}\label{eq:EA-q}
  q'=-q^2+\frac\zeta r-q\psi.
\end{equation}
We have obtained the boundedness of $\zeta/r$ in Proposition
\ref{prop:zeta}. The boundedness of $\psi$ can also be derived as
follows.
\begin{proposition}[Boundedness of $\psi$]\label{prop:psi}
  $\psi$ is bounded above and below by
  \[0<\psim\leq\psi(r,t)\leq\psiM,\quad\forall~(r,t)\in
    [0,D]\times\R_+.\]
  where $\psim$ is defined in \eqref{eq:nu}, and
  $\psiM:=\|\phi\|_{L^\infty}\|\rho_0\|_{L^1}$.
\end{proposition}
\begin{proof}
  The upper bound can be simply obtained by
  \[\psi(r,t)=\int_{\R^n}\phi(r\e_1-y)\rho(y,t)\,dy\leq\|\phi\|_{L^\infty}\|\rho_0\|_{L^1}=:\psiM.\]
  For the lower bound, using the a priori bound on the support
  \eqref{eq:flocking}, the decreasing property of $\phi$, and the
  definition of $\nu$ in \eqref{eq:nu}, we get
  \[\psi(r,t)=\int_{|y|\leq D}\phi(r\e_1-y)\rho(y,t)\,dy\geq
    \phi(2D)\int_{|y|\leq D}\rho(y,t)\,dy=\nu.\]
\end{proof}

Now, we are ready to discuss threshold conditions on $q$. First, we
state a rough result, making use of the boundedness on $\zeta/r$ and $\psi$.

\begin{proposition}[Rough threshold conditions on
  $q$]\label{prop:EA-q-rough}~
  
  \begin{itemize}
    \item (Subcritical region) Let $C_0\leq\frac{\psim^2}{4}$. If
  $q_0\geq\frac12\left(-\psim-\sqrt{\psim^2-4C_0}\right)$, then $q(t)$ stays bounded
  in all time.
    \item (Supercritical region) If
      $q_0<\frac12\left(-\psiM-\sqrt{\psiM^2+4C_0}\right)$, then
      $q(t)\to -\infty$ in finite time.
  \end{itemize}
\end{proposition}
\begin{proof}
  The results follows from simple comparison principles. We will only
  show the subcritical region.

  The upper bound on $q$ is trivial. If $q(t)\geq\sqrt{C_0}$, then
  \[q'(t)\leq-q^2(t)+C_0-\psim~q(t)<0.\]
  This directly implies $q(t)\leq\max\{q_0, \sqrt{C_0}\}$.
  
  For the lower bound, we will show that $q(t)\geq
  \frac12\left(-\psim-\sqrt{\psim^2-4C_0}\right)$, by
  contradiction.
  Suppose $q$ does not have such lower bound. Then, there exists a
  time $t_0$ such that $q(t_0)=
  \frac12\left(-\psim-\sqrt{\psim^2-4C_0}\right)$
  and $q'(t_0)\leq0$. On the other hand, we compute
  \[q'(t_0)> -q^2(t_0)-C_0-\psim~q(t_0)=0.\]
  This leads to a contradiction.
\end{proof}

The thresholds conditions are not sharp, due to the lack of precise
control of the non-locality. However, in the special case when $\phi$
is a constant, we have $\zeta=0$ and $\psim=\psiM$. Then, Proposition
\ref{prop:EA-q-rough} becomes sharp.

The threshold conditions can be improved, if we take into account of
the fast decay property of $\zeta/r$.
The idea is to the dynamics as the following autonomous system
\begin{equation}\label{eq:EA-qV}
  \begin{cases}
    q'=-q^2-c_1q+c_2B,\quad c_1\in[\psim,\psiM],~~c_2\in[-1,1],\\
    \frac{d}{dt}B=-\nu B,
  \end{cases}\quad
  \begin{cases}
    q(0)=q_0,\\ B(0)=C_0.
  \end{cases}
\end{equation}
Then, perform a phase plane analysis on \eqref{eq:EA-qV} assuming
$c_1$ and $c_2$ are constant. Finally, establish a comparison
principle to obtain threshold conditions for \eqref{eq:EA-qV}.
Following directly from \cite[Theorem 5.1]{tadmor2014critical}, we
have the following enhanced threshold conditions.

\def\sqp{\sigma_q^+}
\def\sqm{\sigma_q^-}
\begin{proposition}[Enhanced threshold conditions on
    $q$]\label{prop:EA-q-enhanced}~
  
   \begin{itemize}
    \item (Subcritical region) There exists a function $\sqp:\R_+\to[-\psim,\infty)$,
      defined as
      \begin{equation}\label{eq:sqp}
\sqp(0)=-\psim,\quad \frac{d}{dx}\sqp(x)=
\begin{cases}
  \displaystyle\frac{1}{2\psim},&x\to0+\\
  \displaystyle\frac{-\sqp(x)^2-\psim \sqp(x)-x}{-\nu x}& \text{if}~\sqp(x)<0\\
 \displaystyle\frac{-\sqp(x)^2-\psiM \sqp(x)-x}{-\nu x}& \text{if}~\sqp(x)\geq0
\end{cases}
\end{equation}
      such that, if
      $q_0\geq\sqp(C_0)$, then $q(t)$ stays bounded  in all time.
    \item (Supercritical region) There exists a function $\sqp:\R_+\to[-\infty,-\psiM)$,
      defined as
\begin{equation}\label{eq:sqm}
\sqm(0)=-\psiM,\quad \frac{d}{dx}\sqm(x)=
\begin{cases}
\displaystyle-\frac{1}{\psiM+\nu},& x\to0+\\
\displaystyle\frac{-\sqm(x)^2-\psiM \sqm(x)+x}{-\nu x}
&x>0.
\end{cases}
\end{equation}
such that, if 
      $q_0<\sqm(C_0)$, then $q(t)\to -\infty$ in finite time.
  \end{itemize}
\end{proposition}

\begin{remark}
 The threshold functions $\sigma_{q+}$ and $\sigma_{q-}$ only depends
 on $\psim$ and $\psiM$. Figure \ref{fig:EA-q} shows an example of the
 thresholds, with $\psim=.8$ and $\psiM=1$.
 One can clearly see that the enhanced threshold conditions are much
 stronger than the rough conditions in Proposition
 \ref{prop:EA-q-rough}, particularly for the subcritical region.
\end{remark}

\begin{figure}[t]
  \includegraphics[width=.7\linewidth]{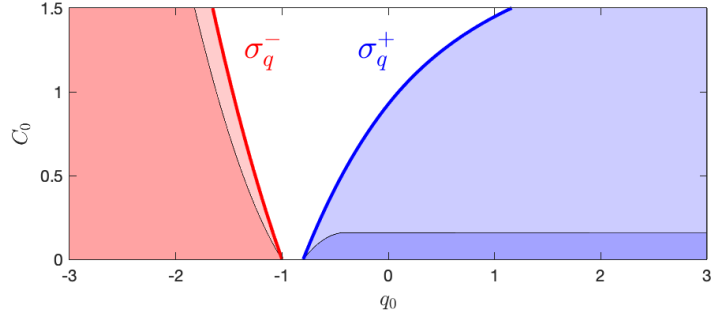}
  \caption{An illustration of the threshold regions for $(q_0, C_0)$
    with parameters $\psim=.8, \psiM=1$.
    Darker areas represent the rough conditions.}\label{fig:EA-q}
\end{figure}

\subsection{Critical thresholds in multi-dimensions}

We are ready to control $\rho$ and $G$.
In 1D, $G_0\geq0$ is the sufficient and necessary condition to insure
global regularity. It is not the case in multi-dimension, due to the
effect of the spectral gap.
Recall the dynamics of $G$
\[G'=-G^2+\psi G-(n-1)\frac{\zeta}{r}.\]
A similar argument as Proposition \ref{prop:EA-q-rough} would yield the
following rough conditions.
\begin{proposition}[Rough threshold conditions on
  $G$]\label{prop:EA-G-rough}~
  
  \begin{itemize}
    \item (Subcritical region) Let $C_0\leq\frac{\psim^2}{4(n-1)}$. If
  $G_0\geq\frac12\left(\psim-\sqrt{\psim^2-4(n-1)C_0}\right)$, then $G(t)$ stays bounded
  in all time.
    \item (Supercritical region) If
      $G_0<\frac12\left(\psiM-\sqrt{\psiM^2+4(n-1)C_0}\right)$, then
      $G(t)\to -\infty$ in finite time.
  \end{itemize}
\end{proposition}
\begin{remark}
 In the special case when $\phi$ is a constant, we recover the sharp
 threshold: global wellposedness if and only if  $G_0\geq0$.
\end{remark}

\begin{remark}
  Let us compute the bound that $e_0=\div\u+\psi=G_0+(n-1)q_0$ has to
  satisfy,  using the rough subcritical conditions on $G_0$ and $q_0$
  \[e_0\geq
    -\frac{n-2}{2}\psim-\frac12\sqrt{\psim^2-4(n-1)C_0}-\frac{n-1}{2}\sqrt{\psim^2-4C_0}.\]
  In particular, for $n=2$, $e_0\geq-\sqrt{\psim^2-4C_0}$, which can
  be picked to be negative. Therefore, the subcritical region is much
  larger than \cite[Theorem 2.1]{he2017global}, which requires a
  tougher smallness condition on $C_0$, as well as $e_0\geq0$.
  Further improvement can be made by enhanced threshold conditions,
  stated in Propositions \ref{prop:EA-q-enhanced} and \ref{prop:EA-G-enhanced}.
\end{remark}

Next, we obtain enhanced threshold conditions on $G$, taking advantage
of the fact that $\zeta/r$ decays exponentially in time.
The result is similar to Proposition \ref{prop:EA-q-enhanced}, as 
the dynamics of $G$ also falls into a similar format as \eqref{eq:EA-qV}
\begin{equation}\label{eq:EA-GV}
  \begin{cases}
    G'=-G^2+c_1 G+c_2B,\quad c_1\in[\psim,\psiM],~~c_2\in[-(n-1),n-1],\\
    \frac{d}{dt}B=-\nu B,
  \end{cases}\quad
  \begin{cases}
    G(0)=G_0,\\ B(0)=C_0.
  \end{cases}
\end{equation}

We state the enhanced threshold conditions as follows. The thresholds
are illustrated in Figure \ref{fig:EA-G}. The regions are much larger
than the rough conditions.

\def\sGp{\sigma_G^+}
\def\sGm{\sigma_G^-}
\begin{proposition}[Enhanced threshold conditions on
    $G$]\label{prop:EA-G-enhanced}~
  
   \begin{itemize}
    \item (Subcritical region) There exists a function $\sGp:\R_+\to[-\psim,\infty)$,
      defined as
      \begin{equation}\label{eq:sGp}
\sGp(0)=0,\quad \frac{d}{dx}\sGp(x)=
\begin{cases}
\displaystyle\frac{n-1}{2\psim},&x\to0+\\
\displaystyle\frac{-\sGp(x)^2+\psim \sGp(x)-(n-1)x}{-\nu x}& x>0
\end{cases}
\end{equation}
      such that, if
      $G_0\geq\sGp(C_0)$, then $G(t)$ stays bounded  in all time.
    \item (Supercritical region) There exists a function $\sGp:\R_+\to[-\infty,-\psiM)$,
      defined as
\begin{equation}\label{eq:sGm}
\sGm(0)=0,\quad \frac{d}{dx}\sGm(x)=
\begin{cases}
\displaystyle-\frac{n-1}{\psiM+\nu},& x\to0+\\
\displaystyle\frac{-\sGm(x)^2+\psiM \sGm(x)+(n-1)x}{-\nu x}
&x>0.
\end{cases}
\end{equation}
such that, if 
      $G_0<\sGm(C_0)$, then $G(t)\to -\infty$ in finite time.
  \end{itemize}
\end{proposition}

\begin{figure}[t]
  \includegraphics[width=.7\linewidth]{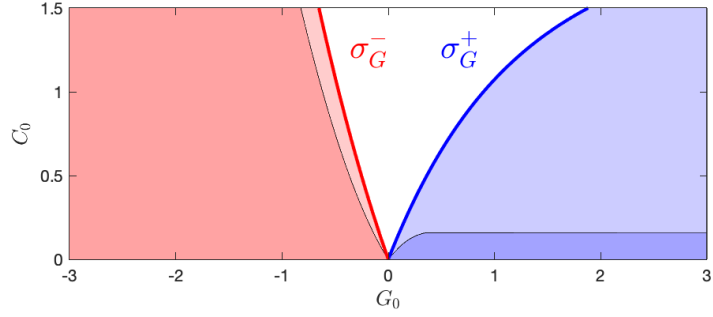}
  \caption{An illustration of the threshold regions for $(G_0, C_0)$
    with parameters $\psim=.8, \psiM=1, n=2$.
    Darker areas represent the rough conditions.}\label{fig:EA-G}
\end{figure}

\begin{remark}
  The threshold curves $\sGp$ and $\sGm$ are dimension dependent.
  In the case $n=1$, one can check that $\sGp\equiv0$ and
  $\sGm\equiv0$. It recovers the sharp critical threshold condition in
  1D, stated in Theorem \ref{thm:EA-1D}.

  For $n\geq2$, we have $\sGp(x)>0$ and
  $\sGm(x)<0$ for $x>0$. There is a gap between the two regions, due
  to the nonlocal effect.
  The gap becomes larger as $n$ increases, illustrated in Figure
  \ref{fig:EA-Gn}. There is no gap when $C_0=0$ (when $\phi$ is a
  constant), regardless of the dimension.
\end{remark}

\begin{figure}[t]
  \includegraphics[width=.7\linewidth]{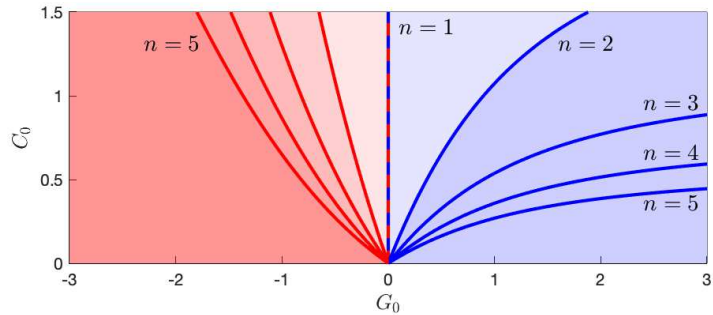}
  \caption{An illustration of threshold curves $\sGp$ and $\sGm$ in different dimensions
    $n=1,\cdots, 5$, with parameters $\psim=.8, \psiM=1$.}\label{fig:EA-Gn}
\end{figure}

Finally, we wrap up the proof of Theorem \ref{thm:EA}.

\begin{proof}[Proof of Theorem \ref{thm:EA}]
 For subcritical initial data, applying Propositions
 \ref{prop:EA-q-enhanced} and \ref{prop:EA-G-enhanced}, we obtain
 the boundedness of $q$ and $G$.
 As $\psi$ is bounded (Proposition \ref{prop:psi}), we get
 $p=G-\psi$ is also bounded.
 Then, Proposition \ref{prop:equivgradu} implies the boundedness of
 $\grad\u$, and global wellposedness is the direct consequence of
 Theorem \ref{thm:EAlocal}. The asymptotic flocking behavior follows
 from Theorem \ref{thm:flock}.

 For supercritical initial data, from Propositions
 \ref{prop:EA-q-enhanced} and \ref{prop:EA-G-enhanced}, we deduce that
 either $q$ or $G$ blows up (which is equivalent to $p$ blows up due
 to the boundedness of $\psi$). Therefore, $\grad\u$ becomes
 unbounded, resulting a loss of regularity in finite time.  
\end{proof}

\section{Further discussion}\label{sec:conclusion}
In this paper, we introduce a new pair of quantities $(u_r,
\frac{u}{r})$, which serve as a nice replacement of the 1D quantity
$\pa_xu$, for pressure-less Eulerian dynamics in multi-dimensions
with radial symmetry.
The applications to the Euler-Poisson equations and the
Euler-alignment equations show significant advantages of studying the
dynamics of the pair, compared to the spectral dynamics (on eigenvalues
of $\grad\u$), as well as the divergence $\div\u$.
The idea has the great potential to be applied to a large class of Eulerian dynamics
with different forces.

There are several possible extensions.
\begin{enumerate}
\item \textit{Systems with pressure.}
  Pressure appears naturally in many models of Eulerian dynamics.
  For the 1D Euler equation with isentropic pressure (known as the
  $p$-system), the Riemann invariants are introduced to handle the
  pressure. The quantities that are relevant to global regularity are
  $\pa_x(u\pm c(\rho))$, where $c(\rho)$ is the sound speed. Global
  regularity has been shown for the $p$-system \cite{chen2015optimal}
  and the Euler-Poisson equations with pressure
  \cite{tadmor2008global} in 1D.
  Global regularity in multi-dimensions is largely unknown. It is
  interesting to understand which quantities serve as a nice
  replacement of $\pa_x(u\pm c(\rho))$ in multi-dimensions with radial symmetry.
\item \textit{Radially symmetric flow with swirl.}
  Radially symmetric solutions can allow swirls. For instance, in 2D,
  $u(\x)=\frac{\x}{r}u(r)+\frac{\x^\perp}{r}R(r)$.
  It is known that rotations can prevent singularity formation
  \cite{liu2004rotation}. Our global regularity result has the potential to be extended
  to radially symmetric data with swirl.
\item \textit{Perturbation around a radially symmetric solution.}
 One next step is to study a non-symmetric perturbation around the
 radially symmetric solution. This would allow us to extend the result
 to a larger class of solutions. 
\end{enumerate}

We leave all these intriguing problems for further investigation.

\bibliographystyle{plain}
\bibliography{Eulerian}

\end{document}